\documentclass[12pt]{article}
\usepackage{graphicx}
\usepackage{natbib} 
\usepackage{url} 
\usepackage{enumitem}
\usepackage[utf8]{inputenc}
\usepackage{amssymb, psfig, epsfig, amsmath, multirow, caption,algorithm,algpseudocode}
\usepackage[colorlinks = true, citecolor = blue, urlcolor = blue]{hyperref}
\usepackage{graphicx}
\usepackage{subfiles}
\usepackage{epstopdf}
\usepackage{float}
\usepackage{booktabs,xltabular,float,subfig}
\usepackage[english]{babel}
\usepackage{amsthm}
\DeclareGraphicsExtensions{.jpg,.fig,.eps,.png}

\newtheorem{remark}{Remark}%
\newtheorem{assumption}{Assumption}
\newtheorem{lemma}{Lemma}
\newtheorem{theorem}{Theorem}
\newtheorem{corollary}{Corollary}
\newtheorem{proposition}[theorem]{Proposition}%
\newtheorem{definition}{Definition}%

\DeclareGraphicsExtensions{.jpg,.fig,.eps}

\newcommand{\EE}{\mathbb{E}}
\newcommand{\PP}{\mathbb{P}}

\newcommand{\ind}[1]{{\bf 1}\left(#1\right)}

\newtheorem{thm}{Theorem}
\DeclareMathOperator*{\argmax}{arg\,max}
\newcommand{\blind}{0}

\begin{document}

\def\spacingset#1{\renewcommand{\baselinestretch}%
{#1}\small\normalsize} \spacingset{1}

\date{}

\if0\blind
{
  \title{\bf Steady-State Analysis and Online Learning for Queues with Hawkes Arrivals}
  \author{Xinyun Chen\\
  	School of Data Science, Chinese University of Hong Kong (Shenzhen),\\
	Guiyu Hong\\
	School of Data Science, Chinese University of Hong Kong (Shenzhen)
}
  \maketitle
} \fi

\if1\blind
{
  \bigskip
  \bigskip
  \bigskip
  \begin{center}
    {\LARGE\bf Title}
\end{center}
  \medskip
} \fi

\bigskip
\begin{abstract}
	We investigate the long-run behavior of single-server queues with Hawkes arrivals and general service distributions and related optimization problems. In detail, utilizing novel coupling techniques, we establish finite moment bounds for the stationary distribution of the workload and busy period processes. In addition, we are able to show that, those queueing processes converge exponentially fast  to  their stationary distribution. Based on these theoretic results, we develop an efficient numerical algorithm to solv the optimal staffing problem for the Hawkes queues in a data-driven manner. Numerical results indicate a sharp difference in staffing for Hawkes queues, compared to the classic GI/GI/1 model, especially in the heavy-traffic regime.
\end{abstract}

\noindent%
{\it Keywords:}  steady-state analysis; Hawkes processes; capacity sizing problem; online learning
\vfill

\newpage
\spacingset{1.25} 
	\section{Introduction}\label{sec: introduction}
As one of the most fundamental queueing model, the single-server queue is used to describe service systems with scarce but reusable recourse. Recent empirical studies found that arrivals in many real queueing systems exhibit a clustering or self-exciting behavior; that is, an arrival may  increase the possibility of new arrivals. In some cases, such clustering behavior is intrinsic to the underlying system. For example, in the stock market, it is a common practice to split a large order into small child orders to reduce transaction cost. As a consequence, one observed arriving order may be followed by a sequence of other child orders \citep{abergel2015long}. As a natural extension of the classic Poisson process, Hawkes process has been used to model arrivals with self-excitement such as order flow in stock market \citep{abergel2015long}, infected patients  during pandemic \citep{bertozzi2020challenges}, and the internet traffic in social media \citep{zhao2015seismic}.


To understand the impact of self-excitement in the arrival process on the long-run performance of service systems, \cite{koops2018infinite} and \cite{daw2018queues} provided analytic solutions to steady-state moments on the number of people in system for different infinite-server systems with Hawkes arrivals.  In particular, \cite{daw2018queues} study the systems with Markovian Hawkes arrivals and phase-type/deterministic service times, whereas \cite{koops2018infinite} study the cases
with non-Markovian Hawkes arrivals and exponential service times. Because of the dependence between customer arrivals and sojourn times,
it is difficult to obtain analytic results for finite-server queues with Hawkes arrivals, including the most fundamental single-server queue (see the dicussion on p.941 of \cite{koops2018infinite}). Instead, simulation algorithm was developed in \cite{perfect2020} to generate unbiased samples from the steady-state distribution of Hawkes$/GI/1$ queue.

Although analytic expression of the steady-state distribution might not be available for finite-server queues with Hawkes arrivals, it is still possible to carry out theoretic analysis for better understanding on  the long-run behavior of such systems and to obtain useful application implications.  In this paper, we establish theoretic results on the steady-state moments and rate of convergence to stationarity for a general class of Hawkes$/GI/1$ queues with non-Markovian Hawkes arrivals and general distributed service times. As an application example,  we then develop an efficient numerical algorithm for performance optimization of Hawkes$/GI/1$ systems under the guidance of the theoretic results.

In detail, we first construct an explicit representation of stationary queueing functions for the Hawkes$/GI/1$ under certain tail conditions on the self-excitation function of the Hawkes process and the service time distribution. Then, we establish existence of higher-order moments for the stationary workload process and busy period. Finally, we obtain the main theoretic result in this paper, showing that the queueing processes of Hakes$/GI/1$ queue converge to their stationary distributions exponentially fast in total variation norm. Our proof techniques root in the cluster representation of Hawkes process \citep{koops2018infinite}. It first enables us to dominate the stationary Hawkes$/GI/1$ system with an auxiliary $GI/GI/1$ system and thus establish the moment bounds. For the ergodicity analysis, we utilize a coupling approach to obtain a finite-time bound and more explicit characterization on the convergence rate. Previous works on convergence rate analysis for queueing models usually resort to the so-called synchronous coupling, e.g. \cite{chen2023online} for $GI/GI/1$ queue, \cite{blanchet2020rates} and \cite{BanerjeeBudhiraja2020} for reflected Brownian motions. In synchronous coupling, a transient system is coupled with a stationary system such that they have independent initial system states but share the same arrival and service processes. In our setting, however, the initial system state is not independent of future arrivals due to the self-exciting behavior in Hawkes arrivals, and thus synchronous coupling is not applicable. To tackle this issue, we construct a different \textit{semi-synchronous coupling} leveraging the cluster representation of Hawkes process. In particular, we split the Hawkes arrivals, via the cluster representation, into two groups - those dependent on the initial state and those independent of the initial state - so that the coupled system shared the same arrivals that are independent of the initial state. We believe the \textit{semi-synchronous coupling} technique is of independent research interest and can be used for analysis of stochastic systems with Hawkes or other auto-correlated input processes.

As an application example, we then develop an online learning algorithm for optimal staffing/capacity sizing of Hawkes$/GI/1$ system under the guidance of the theoretic results we have developed. The algorithm is designed using the framework proposed in \cite{chen2023online} for which the choice of algorithm hyper-parameters and the performance guarantee rely highly on the ergodicity behavior of the underlying queueing system. Based on the exponential convergence result, we are able to design an online learning algorithm with logarithmic regret bound for optimizing service capacity in  Hawkes$/GI/1$ systems. Utilizing this algorithm as an efficient numerical tool, we then carry out an empirical study on the impact of self-excitement in arrival process to decision making in service systems. The numerical results illustrate that ignoring self-excitement or auto-correlation in the arrival data could lead to significant under-staffing in service systems. In addition, we design a set of numerical experiments to investigate Hawkes$/GI/1$ queue in heavy traffic. In the literature, the ``square root rule" is a well-known rule of thumb for optimal capacity sizing or staffing queues in heavy traffic, see for example, \cite{garnett2002designing,whitt2004efficiency,zeltyn2005call,LeeWard2019} and the references therein. With the square root staffing rule, the system would usually enjoy the efficient use of the servers but still with a quality guarantee of the service at the same time. Our numerical results, however, show that the square root rule does not hold in this example of service capacity sizing for Hawkes$/GI/1$ queues, indicating different asymptotic behavior of Hawkes$/GI/1$ queues in the heavy traffic. We believe these numerical findings not only bring application insights to decision makings in service system with self-exciting arrivals, but could also potentially lead to some new theoretic development.


\paragraph{Organization of the paper.} In Section \ref{sec: literature}, we review the related literature. In Section \ref{sec: preliminaries}, we introduce the preliminaries of Hawkes process and our main model, the Hawkes$/GI/1$ queue. The main theoretic results are presented in Section \ref{sec: main results}. In Section \ref{sec: stationary haweks queue}, we introduce our key building blocks, namely, the construction of the stationary version of Hawkes queue and the auxiliary $M/GI/1$ queue. Then we give the proof of the moment bounds for stationary Hawkes queues. In Section \ref{sec: ergodicity}, we show that the Hawkes$/GI/1$ queue converges to its steady state exponentially fast. As an application example, we introduce the optimal service capacity sizing problem in Section \ref{sec: online learning} and develop the online learning algorithm along with the regret bound analysis. Numerical results are reported in Section \ref{sec: num}. Finally, we conclude the paper in Section \ref{sec: conclusion}.
\section{Related Literature}\label{sec: literature}
The present paper is related to the following two streams of literature.
\paragraph{Queues with Hawkes arrivals.} Our paper relates to literature about queues with Hawkes arrivals. Hawkes process is initially introduced by sequel papers of Alan Hawkes in \citep{hawkes1971point,hawkes1971spectra,hawkes1974cluster} to model the occurrence of earthquakes in seismology and epidemics. \cite{gao2018functional} applies heavy-traffic analysis to Hawkes$/GI/\infty$ and shows that the properly scaled queue length process converges to a Gaussian process as background intensity grows, whose covariance kernel depends on the excitation function and the service times, which has no closed form in general. \cite{koops2018infinite} investigates the number of customers in a Hawkes$/M/\infty$ queue and applies analytical methods to provide transient moment bounds of the number of customers in the system. \cite{daw2018queues} analyzes Hawkes$/PH/\infty$ and Hawkes$/D/\infty$ models with exponential excitation function and obtains the transient and stationary moment bounds of the number of customers in the system. As it is difficult to obtain analytic solution to finite-server systems with Hawkes arrivals, \cite{perfect2020}  proposes a perfect simulation algorithm to generate samples exactly from the steady-state of Hawkes$/GI/1$ queue. 

\paragraph{Staffing with Non-standard Arrival Processes.} Our paper also relates to a small and emerging literature on staffing service systems with non-standard arrival processes. \cite{zhang2014scaling} uses a doubly stochastic Poisson process (DSPP) with a rate driven by a Cox-Ingersoll-Ross (CIR) process to model the overdispersion of the arrivals in a call center. The authors also suggest to use safety staffing in accordance to the level of  overdispersion. \cite{sun2021staffing} also consider CIR driven DSPP as arrival input and suggest that the square-root staffing calibrated by the overdispersion level could achieve probability of delay target in heavy-traffic. \cite{heemskerk2022staffing} investigate the staffing problem for many-server queue with Cox arrival process via an offered-load approach using analytic results for infinite-server queue with the same arrival process as developed in \cite{heemskerk2017scaling}. \cite{daw2019staff} considers the staffing problem for queues with batch arrivals and finds out that it is asymptotically optimal to set the service capacity by mean arrival rate plus additional safety staffing level proportional to batch sizes as batch sizes grow to infinity. Here we consider the staffing problem under Hawkes arrivals and develop a numerical solution that works in a data-driven manner.
\section{Preliminaries and Main Model}\label{sec: preliminaries}
In this section, we introduce the Hawkes process considered and the corresponding Hawkes$/GI/1$ queue. Specifically, we first introduce the Hawkes process in Section \ref{subsec: hawkes basics} and then its stationary version in Section  \ref{subsec: stationary Hawkes}. Finally, we introduce our main model, i.e., the Hawkes$/GI/1$ queue and its corresponding stochastic processes in Section \ref{subsec: hawkes queue assumptions}.
\subsection{Hawkes Process}\label{subsec: hawkes basics}
Mathematically, a Hawkes process is a counting process $N(t)$ that satisfies
$$
\mathbb{P}(N(t+\Delta t)-N(t)=n \vert \mathcal{F}(t))=
\begin{cases}
\lambda(t)\Delta t + o(\Delta t), &n=1\\
o(\Delta t), &n>1\\
1-\lambda(t)\Delta t +o(\Delta t), &n=0,
\end{cases}
$$
as $\Delta t\rightarrow0$. Here $\mathcal{F}(t)$ is the associated filtration, and $\lambda(t)$ is the conditional intensity that satisfies 
\begin{equation}\label{eq:def hawkes}
\lambda(t)=\lambda_0+\sum_{i=1}^{N(t)}h(t-t_i).
\end{equation}
Here $\lambda_0$ is the background intensity and $t_1,\  t_2,\cdots$ are the time of arrivals, and $h(\cdot): \mathbb{R}^+\rightarrow\mathbb{R}^+$ is called the excitation function. From equation \eqref{eq:def hawkes}, the Hawkes process has self-exciting property and once an arrival happens, the arrival rate will grow instantly.
\par Following \cite{koops2018infinite}, we introduce an equivalent cluster representation of Hawkes process. This representation basically describes Hawkes process as a kind of branching-process with immigrants, and is useful in our construction of stationary Hawkes process and analysis on the stationary Hawkes/GI/1 queue.  
\begin{definition}[Cluster representation of Hawkes Process]\label{def: cluster hawkes}
Consider a time (possibly infinite) $T\geq0$, and define a sequence of random events and their corresponding arrival time $\{t_n<T:n\geq1\}$ according to the following procedure:
\begin{enumerate}
	\item Consider a set of immigrants that arrives at time $\{t_l^1\leq T : l\geq1\}$, according to a time homogeneous Poisson process on $[0,T]$ with rate $\lambda_0$.
	\item For each immigrant time $t_l^1,\ l\geq1$, define a cluster $C_l$, which is the set for the arrivals brought by immigrant $l$. We index the events in $C_l$ by $k\geq1$ and represent the events by $e_l^k=(k,t_l^k,p_l^k)$. Here $t_l^k$ is the arrival time of this event and $p_l^k$ is the index of the parent event of this event. By this representation, we represent the $l$-th immigrant event as $e_l^1=(1,t^1_l,0)$.
	\item The cluster $C_l$ is generated by a branching process. Initialize $k=1$ and $C_l=\{e_l^1\}$. For events in $C_l$, let $n=\vert C_l\vert $, i.e., the current cardinal of cluster $C_l$. Generate the descendant events of next generation $e_l^{n+1}=(n+1,t_l^{n+1},k),\cdots,e_l^{n+\Lambda}=(n+\Lambda,t_l^{n+\Lambda},k)$ following inhomogeneous Poisson proecess on $[t_l^k,T]$ with rate function $\lambda(t)=h(t-t_l^k)$. Update $k\leftarrow k+1$ and combine $e_l^{n+1},\cdots,e_l^{n+\Lambda}$ into the $C_l$. Keep this procedure until no new events happens.
	\item Collect the arrival times of all events from all cluster as $\{t_n:n\geq1\}=\cup_l\{t_l^k, k=1,\cdots,\vert C_l\vert \}.$
\end{enumerate} 
\end{definition}
Let $m=\int_0^\infty h(t)dt$. According to Definition \ref{def: cluster hawkes}, $m$ is the expected number of events of the next generation that can be generated by a single event. For those events $e_l^k$ that are not an immigrant, let us define their birth times as $b_l^k=t_l^k-t_l^{p_l^k}$. Conditional on $\vert C_l\vert$, by the property of inhomogeneous Poisson distribution, $\{b_l^k:l\geq1, k\geq2\}$ are i.i.d. positive random variables with probability density function $f(t)=h(t)/m$ for $t\geq0$. The arrival time of the first event $t^1_l$ is called the arrival time of cluster $C_l$. The arrival time of the last event is denoted by $\delta_l\equiv \max_k t_l^k$ and is called the departure time of cluster $C_l$. In the rest of the paper,  we will specify the distribution of a Hawkes process by parameter $(\lambda_0,m,f(\cdot))$. 
\subsection{A Stationary Hawkes Process}\label{subsec: stationary Hawkes}
For the Hawkes process with parameters $(\lambda_0,m,f(\cdot))$ to be
stable in the long term, intuitively, each cluster should contain a finite number of events on average.
Therefore, we shall assume that $m < 1$ throughout the paper. This is a common assumption used in the literature and it is known that the Hawkes process has a unique stationary distribution
under this assumption \citep{hawkes1974cluster,bremaud2002rate}. Following \cite{perfect2020}, we can	construct a stationary Hawkes process via cluster representation as follows. 

First, we extend the homogeneous Poisson process $\{t_l^1\}$ of the immigrants, or equivalently, the cluster arrivals, to
time interval $(-\infty,\infty)$. For this two-ended Poisson process, we index the sequence of immigrant
arrival times by $\{\pm1,\pm2,...\}$ such that $t^1_{-1} \leq 0 < t^1_1$ and generate i.i.d. copies of clusters $\{C_{\pm l}: l=1,2,...\}$  for each $l$ following the procedure described in Definition 2. Then, the collection of all the events in $\{C_{\pm l}: l=1,2,...\}$ form a stationary Hawkes process on $(-\infty,\infty)$. In particular, for any pair of $a<b$ on the real line, define 
$$N(a, b)  \triangleq \vert \cup_{l=-\infty}^{\infty} \{(k,t_l^k,p_l^k): k= 1, 2, ..., \vert C_l\vert, a\leq t_l^k\leq b\}\vert,$$
as the number of events in  $\{C_{\pm l}: l=1,2,...\}$ that arrive on $[a,b]$. Then, we have, for any $a<b$ and $s>0$, $N(a,b) \stackrel{d}{=} N(a+s, b+s)$
following the fact that the two-ended Poisson process $\{t^1_{\pm 1}, t^1_{\pm 2},...\}$ is stationary on $(-\infty,\infty)$ and that the clusters $\{C_{\pm l}: l=1,2,...\}$ are i.i.d. and independent of the two-ended Poisson process.

\subsection{The Hawkes$/GI/1$ Model and Assumptions}\label{subsec: hawkes queue assumptions}
We consider a single-server queue where customers arrive according to a \textbf{stationary Hawkes process}
with parameters $(\lambda_0, m, f(\cdot))$ throughout this article. Following Section \ref{subsec: stationary Hawkes}, the stationary Hawkes process is a two-ended process on $(-\infty,\infty)$. A usual way to index customers is by the order of their arrival times. In particular, the $n$-th customer arriving after $0$ is indexed by $n$ while the last 
$n$-th customer arriving before $0$ is indexed by $-n$, for $n=1,2,...$. Note that there is no customer arriving exactly at time 0 almost surely. Then, the corresponding arrival time is denoted as $\{t_{\pm n}:n=1,2,...\}$.  

To facilitate analysis using cluster representation of the Hawkes process, we introduce a second way to index customers. Recall that $\{C_l: l= \pm 1, \pm 2,...\}$ is the set of i.i.d. clusters of the stationary Hawkes process. We denote by a pair of numbers $(l,k)$ as the customer corresponding to the $k$-th event in cluster $C_l$. The corresponding arrival time is $t_l^k$. Note that there is a one-one correspondence between the indices $n$ and $(l,k)$ as illustrated by Figure \ref{fig: job labels}.  Throughout the paper, we will use these two ways of indexing customers alternatively. 

\begin{figure}
\centering
\includegraphics[width=0.6\linewidth]{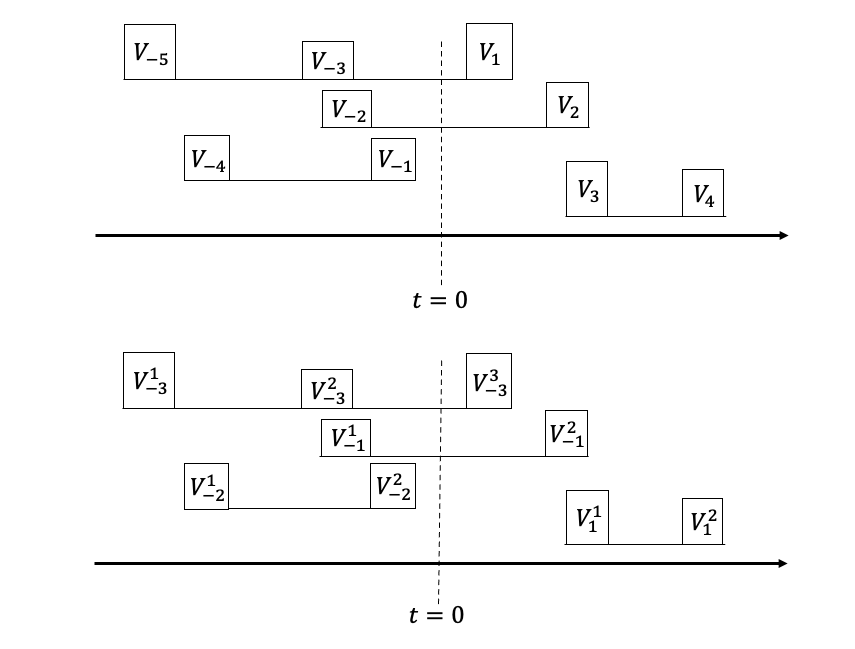}
\caption{Illustration of the job labels. Each rectangle represents an arrival. In the top panel, the jobs are labeled according to the arrival times. In the bottom panel, the jobs are labeled by clusters.}
\label{fig: job labels}
\end{figure}

Upon arrival, customer $n$ (or $(l,k)$) brings to the system a job of size $V_n$ (or $V_l^k$) which follows the distribution of a positive random variable $V$ with continuous pdf $g(\cdot)$. To include job size information, we shall write each event in $C_l$ as $e_l^k = (k,t_l^k,p_l^k,V_l^k)$. 
Customers are served FIFO and the server processes jobs at a constant rate $\mu>0$, i.e., the service time for a customer with job size $V$ is equal to $V/\mu$. Without loss of generality, we assume that $\EE[V]=1$. 

Next, we define the queueing processes of interest. First, let  $W(t)$ be the total amount of workload in the system at time $t$. Given the initial value $W(0)$,  the dynamic of $W(t)$ for $t\geq 0$ can be defined as follows. Let
$N(t)$ be the number of arrivals on time $[0,t)$ and define the process
\begin{equation}\label{eq: R}
R(t) =\sum_{n=1}^{N(t)}  V_n -\mu t\text{ for all }t \geq 0.
\end{equation}
Then, we have
$$dW(t) = dR(t) + dL(t), \text{ with } L(0) = 0, ~dL(t) \geq 0 \text{ and }W(t)dL(t) = 0.
$$
Here the function $L(t)/\mu$ equals to the server’s idle time by
time $t$.  In addition, we have a closed-form expression for $W(t)$ for given $W(0)$ and $R(t)$:
\begin{equation}\label{eq: workload dynamics}
W(t)=W(0) + R(t) -0\wedge\min_{0\leq s\leq t}\left \{W(0)+R(s)\right\}.
\end{equation}
In addition to the workload process, we are also interested in the so-called observed busy period process $X(t)$. In particular, for any $t\geq 0$, $X(t)$ is defined as the time elapsed since the last time when the server is idle, i.e.,
$$X(t)\triangleq t - \sup\{s\leq t: W(s)=0 \}.$$

We close the section by presenting the technical assumptions on the Hawkes$/GI/1$  queue. First, to ensure that the Hawkes$/GI/1$ queue is stable, we impose the following stability condition saying that the average arrival rate of jobs should be smaller than the service rate. 
\begin{assumption}[Stability Condition]\label{assmpt: stable} We assume that there exist {$\varepsilon>0$ and $\underline{\mu}>0$} such that
$$\mu\in\mathcal{B}\equiv[\underline{\mu},\bar{\mu}],\quad \lambda_0\cdot\frac{\EE[V]}{1-m}<\underline{\mu}-\varepsilon.$$
\end{assumption}
In addition, we also impose light-tail conditions on the birth time $b$ of the Hawkes process and individual workload $V$. These conditions will be mainly used in the proof of geometric ergodicity of the Hawkes$/GI/1$ queue. 
\begin{assumption}[Light-tailed Birth Time]\label{assmpt: light tail b}
The birth time $b$ of the Hawkes process 
has a continuous density $f(\cdot)$ and finite moment generating function around $0$, i.e.
$$\psi_b(\theta)\equiv \EE[\exp(\theta b)]<0,$$
for some $\theta >0$.
\end{assumption}


\begin{assumption}[Light-tailed Individual Workload]\label{assume: light-tail of V}
The individual workload $V$ has a finite moment generating function in a neighborhood of the origin, i.e.,
$$\psi_V(\theta)\equiv \EE[\exp(\theta V)]<0.$$
\end{assumption}

\section{Main Results} \label{sec: main results}
In this part, we present the main results on the long-run behavior of the Hawkes$/GI/1$ queue and some proof ideas. In Section \ref{subsec: CFTP construction} we first construct a stationary version of Hawkes$/GI/1$ queue so that we can explicitly characterize the steady-state distributions of the queueing processes. In Section \ref{subsec: moment result}, we present the finite moment results on the steady-state workload and busy periods, which are analog of the results in \cite[Chapter X]{asmussen2003applied} for the $GI/GI/1$ model. Finally, in Section \ref{subsec: ergodicity result}, we present the exponential ergodicity results in the form of exponentially-decay memory on the initial state and explain the coupling idea used in the proofs.

\subsection{A Stationary Hawkes/GI/1 Queue}\label{subsec: CFTP construction}
Suppose Assumption \ref{assmpt: stable} holds. We shall construct a stationary version of the Hawkes/GI/1 queue following the same idea as in \cite{loynes1962}. 
For a two-ended stationary Hawkes process as introduced in Section \ref{subsec: stationary Hawkes}, we index the arrivals after time 0 by $n=1,2,...$, and arrivals before time 0 by $n = -1, -2,\cdots$. Recall that for all $t\geq 0$, we have defined $N(t)$ as the number of arrivals in $[0,t)$. Now, we define $N(-t)$ as the minus of the number of arrivals in $[-t, 0]$. 
For any $t\in\mathbb{R}$ and $s\geq 0$, define
$$R^{\leftarrow}(t, s) \triangleq 
\sum_{n= N(t-s)}^{N(t)} V_n - \mu s.$$
Since the counting process $N(\cdot)$ is stationary and $V_n$ are i.i.d.,  $R^{\leftarrow}(t,
\cdot)$ follows the same distribution as $R^{
\leftarrow}(0,\cdot)$ for all $t\in\mathbb{R}$.	
Under Assumption \ref{assmpt: stable}, for each $t\geq 0$, $R^{\leftarrow}(t,s)$  is a process of negative drift and will go to $-\infty$ as $s\to\infty$. Therefore,  we can define
$$\tilde{W}(t) \triangleq \max_{s\geq 0} R^{\leftarrow}(t, s).$$
The following proposition shows that $\tilde{W}(\cdot)$ is a stationary version of the workload process of the Hawkes$/GI/1$ queue with service rate $\mu$.
\begin{proposition}\label{prop: stationary workload}
Under Assumption \ref{assmpt: stable},   $\{\tilde{W}(t):t\geq 0\}$ is a stationary process. Besides, the dynamic of $\tilde{W}(\cdot)$ follows the workload process of a Hawkes/GI/1 queue with customers arriving according to the stationary Hawkes process $N(t)$ with i.i.d. job sizes $V_n$ and served with rate $\mu$.
\end{proposition}
Similarly, we can also construct a stationary version of the observed busy period. For any $t\geq 0$, define 
\begin{equation}\label{eq: X def}
\tilde{X}(t) =  \argmax_{s\geq 0} R^{\leftarrow}(t,s) .
\end{equation}
\begin{proposition} \label{lmm: stationary X} 
The process $\tilde{X}(\cdot)$ is stationary. In addition, for each $t\geq 0$, $\tilde{X}(t)$ equals the observed busy period corresponding to the workload process $\tilde{W}(\cdot)$. 
\end{proposition}

\subsection{Moment of Stationary Distributions}\label{subsec: moment result}
Following the construction in Section \ref{subsec: CFTP construction}, we shall denote the stationary workload and the stationary observed busy periods of Hawkes$/GI/1$ queue by $\tilde{W}$ and $\tilde{X}$, respectively. Our first main result is on the existence of moments and moment generating function of $\tilde{W}$ and busy time $\tilde{X}$. 
\begin{thm}\label{thm: moment of W}
Under Assumptions \ref{assmpt: stable} and \ref{assmpt: light tail b}, the following statements are true.
\begin{enumerate}
	\item[(a)] If $\EE[V_1^{{k+1}}]<\infty$, then $\EE[\tilde{W}^k]<\infty$ and {$\EE[\tilde{X}^k]<\infty$}.
	\item [(b)] In addition, if Assumption \ref{assume: light-tail of V} holds, then
	$$\EE[\exp(\theta_1\tilde{W})]<\infty,\quad \text{and}\quad {\EE[\exp(\varepsilon\theta_1\tilde{X})]<\infty},$$
	where $\theta_1$ is a positive constant specified in \eqref{eq: theta1 def} and $\varepsilon$ is defined in Assumption \ref{assmpt: stable}.
\end{enumerate}
\end{thm}
Theorem \ref{thm: moment of W} is an analogue of theorem 2.1 in \cite[Chapter X]{asmussen2003applied} and theorem 5.7 in \cite[Chapter VIII]{asmussen2003applied}. For the proof of workload, the main idea is to bound $\tilde{W}(\cdot)$ from above using an M/GI/1 queue following the idea in \cite{perfect2020}. For the proof of the busy period, we connect the busy period with the workload of another Hawkes queue but with a smaller service rate. The proof details including the construction of the $M/GI/1$ queue and the Hawkes queue with a smaller service rate will be given in Section \ref{sec: stationary haweks queue}. 


\subsection{Exponential Ergodicity of Associated Queueing Process}\label{subsec: ergodicity result}

Our second main result is the geometric (exponential) ergodicity of the Hawkes/GI/1 queue.
Following the idea in \cite{onlinequeue2020}, we show that two Hawkes$/GI/1$ queues with different initial states can be properly coupled so that they will converge to each other exponentially fast, i.e., the dependence of the queueing process on the initial state of the system decays exponentially fast. In the setting of Hawkes$/GI/1$ queues, such dependence on the initial state can be attributed to two different sources: (i) the impact of initial congestion level in the queuing system, which is reflected by the initial values of the queueing process; (ii) the impact of Hawkes arrivals before time 0 on future arrivals. The proof method of \cite{onlinequeue2020} based on synchronous coupling can only deal with the first source of dependence in the setting of $GI/GI/1$. To deal with the second source of dependence that is caused by Hawkes arrivals, we shall design a new coupling called \textit{semi-synchronous coupling} utilizing the cluster representation of Hawkes processes introduced in Section \ref{subsec: hawkes basics}.

A key insight from the cluster representation is that at any time point $t$, the impact of Hawkes arrivals in the past on the future only lies in those clusters that arrive before time $t$ and depart after time $t$. To accurately describe this insight and our new coupling method, we need to introduce some extra notations. For each $t\in \mathbb{R}$, we define
$$N_0(t) =  \{(k,t_l^k,p_l^k,V_l^k): (k,t_l^k,p_l^k,V_l^k)\in C_l, s.t.~ t_l^1<t, \delta_l\geq t\},$$
where $\delta_l$ is the departure time of cluster $C_l$ as defined in Section \ref{subsec: hawkes basics}. In other words, $N_0(t)$ is a union of all clusters that arrive before time $t$ but last after time $t$. As the clusters that arrive after time $t$ are independent of clusters that arrive before time $t$, customer arrivals after time $t$ are independent of $W(t)$ conditional on $N_0(t)$.  We call the set $N_0(t)$ the \textit{Hawkes memory} at time $t$. At time $0$, $N_0(0)$ captures all impact of Hawkes arrivals before time 0 to future arrivals, so it is natural to include $N_0(0)$ as part of the initial system state.

Next, we introduce two stochastic processes that are derived from $N_0(\cdot)$ and useful in our proofs. First, we define  the \textit{total residual job size} at time $t$ as 
\begin{equation}\label{eq: J0}
J_0(t) = \sum_{(k,t_l^k,p_l^k,V_l^k)\in N_0(t), t_l^k\geq t} V_l^k,
\end{equation}
i.e., the total size of jobs that will arrive \textit{after} time $t$ from  clusters that have arrived before time $t$.  Then, we define the \textit{residual life time} of past history as 
$$L_0(t) \triangleq \max\{t_l^k: (k,t_l^k,p_l^k,V_l^k)\in N_0(t)\},$$
i.e., the remaining time to the last event of the past history.

Now we are ready to introduce our new coupling design for the Hawkes$/GI/1$ queue. Consider  two Hawkes$/GI/1$ queues with given initial workload $W^i(0)$, system busy time $X^i(0)$ and Hawkes memories $N_0^i(0)$ for $i=1,2$. We couple two systems such that they share the same customer arrivals (including their individual workload) from the clusters that arrive after time $0$, i.e., they share the same sequence of clusters after time $0$. We call this a \textit{semi-synchronous coupling} as the two systems do not exactly share the same arrivals and service times after time $0$ due to customers from $N_0^1(0)$ and $N_0^2(0)$. An illustration is given in Figure \ref{fig: semi-synch}.

\begin{figure}
\centering
\includegraphics[width=0.8\linewidth]{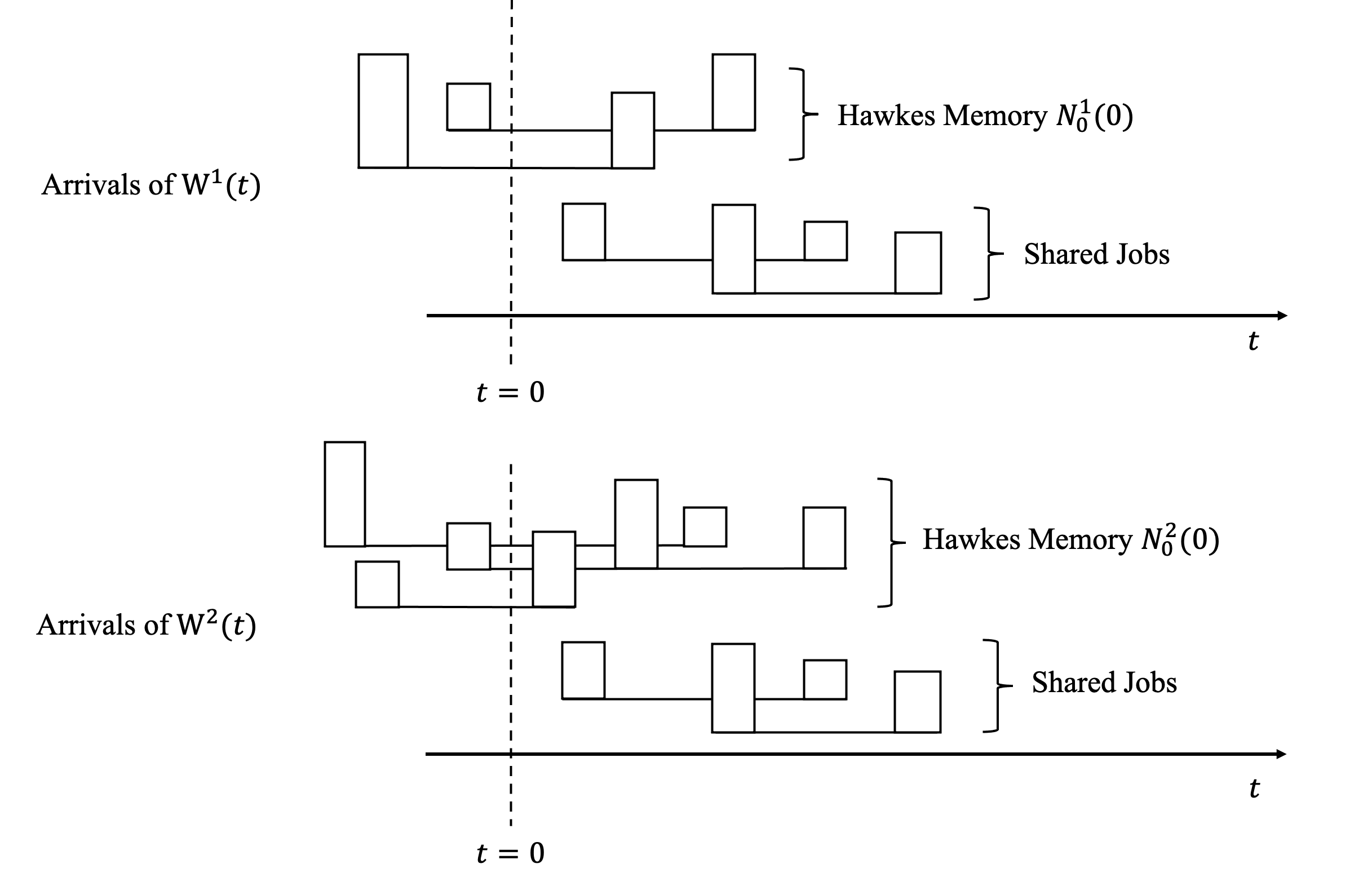}
\caption{Illustration of semi-synchronous coupling. Each rectangle represents a job and the jobs on the same line form a cluster. The Hawkes Memories $N_0^1(0),N_0^2(0)$ are different and the clusters coming after time 0 (shared jobs) are the same.}
\label{fig: semi-synch}
\end{figure}

In the Theorem \ref{thm: geometric ergodicity} below, we show that two Hawkes$/GI/1$ queues that are semi-synchronously coupled will converge to each other exponentially fast. The proof is given in Section \ref{sec: ergodicity}.

\begin{thm}
\label{thm: geometric ergodicity}
Suppose that two Hawkes/GI/1 queues with initial states $(W^1(0),X^1(0), N^1_0(0))$ and $(W^2(0), X^2(0), N^2_0(0))$ are  semi-synchronously coupled and that Assumptions \ref{assmpt: stable} to \ref{assume: light-tail of V} hold. Then, there exists two positive constants $\eta, \theta>0$ such that for all $m\geq 1$
$$
\EE\left[\vert W^1(t)-W^2(t)\vert^m\Big\vert\mathcal{F}_0\right]\leq 
D_0^m\sum_{i=1}^{2}\exp\left(-(\eta\mu t-(\eta+2\theta)\mu L_0^i-\theta(W^i(0)+J^i_0(0)))_+\right),
$$
and
$$\EE\left[\vert X^1(t)-X^2(t)\vert^m\Big\vert\mathcal{F}_0 \right]\leq \left(X^1(0)+X^2(0)+t\right)^m \sum_{i=1}^{2}\exp\left(-(\eta\mu t-(\eta+2\theta)\mu L_0^i-\theta(W^i(0)+J^i_0(0)))_+\right),$$
where $L_0^i(0)$ and $J^i_0(0)$ are the total residual job size and residual lifetime corresponding to $N^i_0(0)$ for $i=1,2$, respectively, and 
$$D_0\equiv |W^1(0)-W^2(0)|+\max(J_0^1, J_0^2),$$
measures the difference between the initial states  $(W^1(0),X^1(0), N^1_0(0))$ and $(W^2(0), X^2(0), N^2_0(0))$.
\end{thm}

\begin{remark}
The value of constants $\eta$ and $\theta$ will be specified in \eqref{eq: theta and eta} in Section \ref{sec: ergodicity}. In fact, this theorem could be modified to bound 
$$\EE[f(|W^1(t)-W^2(t)|)|\mathcal{F}_0]$$
for any monotonic $f(\cdot)$ with $f(0)=0$ by directly replace $D_0^m$ into $f(D_0)$ in the right-hand side.
\end{remark}
As a direct result, we obtain exponential convergence to steady-state distribution (exponential ergodicity) for the workload process and the observed busy period process of Hawkes$/GI/1$ queue (for exponential ergodicity, see \citep{MeynTweedie95Ergo}). Let $\PP^t(\cdot| W(0),X(0),N(0))$ be the probability measure of $(W(t),X(t),N(t))$ conditional on the initial state $(W(0),X(0),N(0))$ and let $\pi$ be the stationary probability measure $\PP^\infty(\cdot)$, then we have the following Corollary \ref{coro: convergence to steady state}.
\begin{corollary}[Exponential Ergodicity]\label{coro: convergence to steady state} 
Under Assumption \ref{assmpt: stable} to \ref{assume: light-tail of V}, we have for any initial state $(W(0),X(0),N_0(0))$,
$$\|\PP^t\left(\cdot|~W(0),X(0),N_0(0)\right)-\pi\|_{TV}<\min\{e^{-\eta\mu t}(M_0+\exp((\eta+2\theta)\mu L_0+\theta (W(0)+J_0(0)))),1\}$$
for a constant $M_0>0$. 
\end{corollary}

\section{Proof of Theorem \ref{thm: moment of W}}\label{sec: stationary haweks queue}
We first introduce an auxiliary M/GI/1 queue to dominate the workload of the Hawkes/GI/1 queue in Section \ref{subsec: mG1}. Then, the analysis on the moments of steady-state workload $\tilde{W}$ and $\tilde{X}$ are given in Section \ref{subsec: workload moments} and \ref{subsec: busy period moments}, respectively.

\subsection{An Auxiliary Dominant $M/GI/1$ Queue}\label{subsec: mG1}
We consider a single-server queue coupled with the Hawkes$/GI/1$ queue such that its customers arrive at $\{\delta_l: l=\pm1, \pm2,\cdots\}$, i.e., departure times of the clusters of the stationary Hawkes process. By reversibility, one can check that $\{\delta_l: l=\pm1, \pm2,...\}$ forms a two-ended Poisson process with rate $\lambda_0$ on $(-\infty,\infty)$. We let the job size brought by customer $l$ equal to
$$S_l \equiv \sum_{k=1}^{\vert C_l\vert} V_l^k, \text{ for }l = \pm 1, \pm 2, ...,$$
i.e., the total amount of workload brought by all customers from cluster $C_l$ of the Hawkes$/GI/1$ queue. Following Definition \ref{def: cluster hawkes}, $\{S_l\}_l$ is an i.i.d. sequence.
In addition, by Assumption \ref{assmpt: stable}, 
$$\lambda_0\EE[S_1]= \lambda_0\frac{\EE[V_1]}{1-m}<\underline{\mu}-\varepsilon<\mu.$$
Therefore, the system as described above is indeed a stable M/GI/1 queue. For all $t\in\mathbb{R}$ and $s\geq 0$, define
$$\hat{R}^{\leftarrow}(t,s) = \sum_{t-s\leq \delta_l<t} S_l -\mu s.$$
Following the same argument as in Proposition \ref{prop: stationary workload}, this $M/GI/1$ queue has a stationary workload process
$$\hat{W}(t) =\max_{s\geq 0} \hat{R}^{\leftarrow}(t,s).$$ 
Similar to $J_0(t)$, for any $t\in\mathbb{R}$, define the \textit{backward total residual job size} as  
\begin{equation}\label{eq: backward J_0}
J^\leftarrow_0(t) = \sum_{(k,t_l^k,p_l^k,V_l^k)\in N_0(t), t_l^k< t} V_l^k.
\end{equation}
Then, the auxiliary $M/GI/1$ queue dominates the Hawkes$/GI/1$ queue in the following sense: the stationary workload $\tilde{W}(t)$ of the Hawkes queue can be bounded by the stationary workload  $\hat{W}(t)$ of the M/GI/1 queue plus the backward total residual job size $J^\leftarrow_0(t)$.
\begin{proposition}\label{prop: mG1}  For any $t\geq 0$, 
\begin{enumerate}
	\item[(a)] $\tilde{W}(t)\leq \hat{W}(t)+J^\leftarrow_0(t)$;
	\item[(b)] $\hat{W}(t)$ and $J^\leftarrow_0(t)$ are independent.
\end{enumerate}
\end{proposition}

\subsection{Moments of the Stationary Workload}\label{subsec: workload moments}
Given Proposition \ref{prop: mG1}, in order to show that $\tilde{W}(0)$ has finite moment, it suffices to show that both  $\hat{W}(0)$ and $J_0^\leftarrow(0)$ have finite moment. Note that $\hat{W}(0)$ is the stationary workload process of an $M/GI/1$ queue and has been extensively studied in the literature (see \cite[Chapter X]{asmussen2003applied}, \cite[Chapter 9]{wolff89} and references therein), so what remains is to bound the moments of $J_0^{\leftarrow}(0)$. We first introduce some technical lemmas.

\begin{lemma}\label{lmm: light-tail workload assumption}
There exist constants $\theta_0,\theta_1>0$ such that
\begin{equation}\label{eq: theta0 def}
	\log\psi_b(\theta_0)= \log \left(\int_0^{\infty}\exp(\theta_0t)f(t)dt\right) < m-1-\log(m),
\end{equation} 
\begin{equation}
	\label{eq: theta1 def}
	\psi_V(\theta_1)<\psi_b(\theta_0/2),\quad\text{and}\quad\frac{\lambda_0}{\underline{\mu}{-\varepsilon}}(\psi_{S_1}(\theta_1)-1)-\theta_1<0,
\end{equation}
where $\psi_V$ and $\psi_{S_1}$ are the moment generating functions of $V_1$ and $S_1$. 
\end{lemma}

\begin{proposition}\label{prop: J0} Under Assumptions \ref{assmpt: stable} and \ref{assmpt: light tail b}, the following statements are true.
\begin{enumerate}
	\item[(a)] If $\EE[V^{n}]<\infty$, $\EE[J_0(0)^n]$ and  $\EE[J_0^\leftarrow(0)^n]<\infty$.
	\item [(b)] If in addition, Assumption \ref{assume: light-tail of V} holds, then
	$$\EE[\exp(\theta_1 J_0(0))],\quad \EE[\exp(\theta_1 J_0^\leftarrow(0))]<\infty,$$
	with $\theta_1$ specified in \eqref{eq: theta1 def}, and $\psi_V(\cdot),\psi_b(\cdot)$ being the m.g.f. of $V_1$ and $b$ respectively.
\end{enumerate}
\end{proposition}

Now we are ready to prove that $\tilde{W}$ has finite moments and moment generating function under proper assumptions on the service time distribution. Following  Proposition \ref{prop: mG1},  we have
$$\tilde{W}(0)^k\leq \left(\hat{W}(0)+J^\leftarrow_0(0)\right)^k\leq 2^{k-1} \left(\hat{W}(0)^k+(J^\leftarrow_0(0))^k\right).$$
Recall that $\hat{W}(0)$ is the stationary workload process of an $M/GI/1$ queue with job size $S_l$ and arrival rate $\lambda_0$. It's known that as long as $\EE[S_l^{k+1}]<\infty$, the stationary workload $\hat{W}(0)$  has finite $k$-th moment (see theorem 2.1 in \cite[Chapter X]{asmussen2003applied}). We compute
$$\EE\left[S_l^{k+1}\right]=\EE\left[\left(\sum_{i=1}^{\vert C_l\vert}V_l^i\right)^{k+1}\right]\leq\EE[\vert C_l\vert ^{k+1}]\EE[V_1^{k+1}],$$
where the last inequality holds as $V_k^l$ are i.i.d and independent of $\vert C_l\vert$. 
Note that $|C_l|$ is the total number of descendants in a Poisson branching process and since $m<1$, $|C_l|$ follows Borel distribution, which has any finite $k$-moment \citep{dwass1969total}. Therefore, $\EE[S_l^{k+1}]<\infty$. 
By Proposition \ref{prop: J0}, $\EE[(J_0^\leftarrow(0))^k]<\infty$.  Therefore,
$$\EE[\tilde{W}(0)^k]\leq 2^{k-1}\left(\EE[\hat{W}(0)^k]+\EE[(J_0^\leftarrow(0))^k]\right)<\infty.$$

We next verify the workload part of statement (b).
By Proposition \ref{prop: mG1}, $\hat{W}(0)$ is independent of $J_0$. So, 
$$\EE[\exp(\theta_1\tilde{W}(0))]\leq \EE[\exp(\theta_1\hat{W}(0)+\theta_1 J^\leftarrow_0(0))]=\EE[\exp(\theta_1\hat{W}(0))]\EE[\exp(\theta_1 J^\leftarrow_0(0))].$$
From Proposition \ref{prop: J0}, we have $\EE[\exp(\theta_1 J^\leftarrow_0(0))]<\infty$. We next show that $\EE[\exp(\theta_1 \hat{W}(0))]<\infty$.
Notice that $\hat{W}(0)$ is the stationary workload of an $M/GI/1$ queue. Then, by transformation version of the Pollaczek-Khinchin formula (see Section 5.1.2 of \cite{daigle2005queueing}), 
$$\EE[\exp(\theta_1\hat{W}(0))]=\frac{(1-\rho)\theta_1}{\theta_1+\frac{\lambda_0}{\mu}(1-\psi_{S_1}(\theta_1))}<\infty,$$
as long as $\rho\equiv\frac{\lambda_0}{(1-m)\mu}<1$ and $\theta_1+\frac{\lambda_0}{\mu}(1-\psi_{S_1}(\theta_1))>0$, which is guaranteed by the definition of $\theta_1$ in \eqref{eq: theta1 def}.
As a consequence, we can conclude 
$\EE[\exp(\theta_1 \tilde{W}(0))]<\infty.$

\begin{remark}
In fact, our procedure shows that under Assumption \ref{assmpt: stable} to \ref{assume: light-tail of V}, any Hawkes queue with service rate $\mu\geq\underline{\mu}-\varepsilon$ has finite m.g.f in the domain $(0,\theta_1]$. Thus, since the $\varepsilon$ is arbitrary, we have that all moments for the stationary workload for the Hawkes queue are finite as long as the Hawkes queue is stable and the job size $V$ has a finite moment generating function around 0.
\end{remark}

\subsection{Moments of Stationary Observed Busy Period}\label{subsec: busy period moments}

It is known that in a stable GI/GI/1 queue, the observed busy period (in terms of how many customers served) has all finite moments under the stationary distribution \citep{nakayama2004finite} if the service times have finite moment generating function around 0. 
In this section, we show that a similar result holds also for the Hawkes$/GI/1$ queue but in terms of real-time. In particular, $\tilde{X}(0)$ has finite moments under certain conditions (observed busy period part of Theorem \ref{thm: moment of W}). However, we could not follow the analysis in \cite{nakayama2004finite} due to the crucial difference between Hawkes models and the $GI/GI/1$ model.


In the proof of the moment bounds of stationary waiting time, we introduce an auxiliary $M/GI/1$ queue, and the upper bound for the $M/GI/1$ queue could give a bound for the workload of Hawkes queues. However, unfortunately, the dominant relationship of workload cannot help the observed busy period due to the existence of $J^\leftarrow_0(t)$. Our solution is to bound $\tilde{X}(0)$ by the last passage time of $\hat{R}^\leftarrow(0,u)$ and connect this last passage time with a stationary workload of another Hawkes queue but with a smaller service rate.

By definition, $\tilde{X}(0)=\argmax_{u\geq 0}R^\leftarrow(0,u)$. Since $\max_{u\geq 0} R^\leftarrow(0,u)\geq 0$, then $\tilde{X}(0)$ can be bounded by the following last passage time of $R^{\leftarrow}(0,u)$, 
$$\tilde{X}(0)\leq \tilde{\tau}\equiv \sup\{u\geq 0:R^\leftarrow(0,u)\geq 0\}.$$

Although the dominance of workload does not imply the dominance of busy time as mentioned ahead, we could find a pathwise connection between $\tilde{\tau}$ and the workload process of another Hawkes queue as follows. Recall that 
$$\tilde{W}(0)=\sup_{u\geq 0}R^\leftarrow (0,u)=\sup_{u\geq 0}\left\{ \sum_{i=-1}^{-N(-u)} V_i -\mu u \right\}$$
is the stationary workload of a Hawkes queue under service rate $\mu$. Similarly, we have
$$\tilde{W}_\varepsilon(0)\equiv\sup_{u\geq 0}\left\{R^\leftarrow (0,u)+\varepsilon u\right\}=\sup_{u\geq 0}\left\{\sum_{i=-1}^{-N(-u)} V_i -(\mu-\varepsilon) u \right\}$$
is the stationary workload of another Hawkes queue with service rate $\mu-\varepsilon$. By Assumption \ref{assmpt: stable}, $\tilde{W}_\varepsilon(0)$ is also well-defined. The following lemma connects $\tilde{\tau}$ with $\tilde{W}_\varepsilon(0)$.
\begin{lemma} \label{lmm: X and W connection}
Suppose Assumption \ref{assmpt: stable} holds, then we have 
$$\tilde{\tau}\leq \tilde{W}_\varepsilon(0)/\varepsilon$$
\end{lemma}

\begin{proof}{Proof of Lemma \ref{lmm: X and W connection}}
By definition, $R^\leftarrow(0,\tilde{\tau})=0$. Using this relation, we have 
\begin{align*}
	\varepsilon \tilde{\tau}&=R^\leftarrow(0,\tilde{\tau})+\varepsilon \tilde{\tau } \leq  \sup_{u\geq 0} \{R^\leftarrow(0,u)+\varepsilon u\}=\tilde{W}_\varepsilon (0).
\end{align*}
Consequently, $\tilde{\tau}\leq \tilde{W}_\varepsilon(0)/\varepsilon$.
\end{proof}

Since $\tilde{W}_\varepsilon(0)$ is the stationary workload of Hawkes queue with service rate $\mu-\varepsilon$, we have proved that it has the corresponding moment bounds in Section \ref{subsec: workload moments}. Then, we are ready to finish the proof of second part of Theorem \ref{thm: moment of W}.
For statement (a), it's clear that 
$$
\EE[\tilde{X}^k(0)]\leq \EE[\tilde{\tau}^k]\leq \EE[\tilde{W}_\varepsilon(0)^k]/\varepsilon^k.
$$
By the first part of Theorem \ref{thm: moment of W}, $\EE[\tilde{W}_\varepsilon(0)^k]<\infty$ as long as $\EE[V_1^{k+1}]<\infty$. This finishes the proof of (a).

For statement (b), following the same analysis,
$$\EE[\exp(\varepsilon\theta_1 \tilde{X}^k(0))]\leq \EE[\exp(\varepsilon \theta_1 \tilde{\tau})]\leq \EE[\exp(\theta_1\tilde{W}_\varepsilon(0))]<\infty$$
where the last inequality follows the proof in Section \ref{subsec: workload moments}.

\section{Proof of Theorem \ref{thm: geometric ergodicity} and Corollary \ref{coro: convergence to steady state}}\label{sec: ergodicity}
We first prove Theorem \ref{thm: geometric ergodicity} in Section \ref{subsec: proof coupling} and then prove Corollary \ref{coro: convergence to steady state} in Section \ref{subsec: proof convergence to steady-state}.
\subsection{Proof of Theorem \ref{thm: geometric ergodicity}}\label{subsec: proof coupling}

We first show that the difference between two semi-synchronously coupled workload processes is always bounded by the initial difference.
\begin{lemma}\label{lmm: sychronous coupling1} Suppose that two Hawkes/GI/1 queues are semi-synchronously coupled  with initial states $(W^1(0),N_0^1(0))$ and $(W^2(0),N_0^2(0))$ respectively. Let $J_0^i(0)$  be the total residual job sizes in $N_0^i(0)$ for $i=1,2$. Then 
$$
\vert W^1(t)-W^2(t)\vert \leq\vert W^1(0)-W^2(0)\vert +\max(J_0^1(0), J_0^2(0))
$$

\end{lemma}

Next, we show that two coupled workload processes $W^1(t)$ and $W^2(t)$ will eventually converge, and we can then give out the ergodic rate based on the ``mixing time". For $i=1,2$, recall that $N^i_0(0)$ is the ``past history" of Hawkes processes and the $L_0^i$ is the last arrival time in $N^i_0(0)$, i.e.,
$$L_0^i = \max\{t_l^k: (k,t_l^k,p_l^k,V_l^k)\in N_0^i(0)\}.$$
Clearly, $L_0^i$ is part of the Hawkes memory $N^i_0(0)$. Define two hitting times
$$\tau_i \triangleq \min\{t: W^i(t) = 0, t\geq L_0^i\}.$$
The following lemma shows that the coupled workload processes $W^1(t)$ and $W^2(t)$ will always equal after these hitting times.
\begin{lemma} Suppose that the conditions in Lemma \ref{lmm: sychronous coupling1} hold. Then, for all $t\geq \max(\tau_1,\tau_2)$, $W^1(t)=W^2(t)$ \ and $X^1(t)=X^2(t)$.
\label{lmm: sychronous coupling2}
\end{lemma}

Given Lemmas \ref{lmm: sychronous coupling1} and \ref{lmm: sychronous coupling2}, we have, for any non-negative and non-decreasing function $f(\cdot)$ on $\mathbb{R}^+$ with $f(0)=0$,
$$f\left(\vert W^1(t)-W^2(t)\vert\right)\leq  f\left(D_0\right)\ind{\max(\tau_1,\tau_2)>t}.$$ 
As a consequence, the analysis of the convergence rate for the Hawkes queues can be reduced to analyzing the tail probability of the hitting times $\tau_1$ and $\tau_2$. For this purpose, we have

Let 
\begin{equation}
(\theta,\eta)=\frac{\min(\theta_0/2,\theta_1)}{6\max\{\bar{\mu},1\}(\eta_1+2\theta_1)}\cdot(\theta_1,\eta_1),
\label{eq: theta and eta}
\end{equation}
with 

\begin{equation}\label{eq: eta1}
\eta_1=\frac{1}{2}\left(\theta_1-\frac{\lambda_0}{\underline{\mu}{-\varepsilon}}(\psi_{S_1}(\theta_1)-1)\right)>0.
\end{equation}

\begin{proposition}\label{prop: mixing time}
Suppose that two Hawkes/GI/1 queues are semi-synchronously coupled  with initial states $(W^1(0),X^1(0),N_0^1(0))$ and $(W^2(0),X^2(0),N_0^2(0))$ respectively. Let $\mathcal{F}_0=\sigma\left(W^i(0),X^i(0),N_0^i(0),i=1,2\right)$, then 
$$\PP(\tau_i>t\vert\mathcal{F}_0)\leq \min\{e^{-\eta \mu t}e^{(\eta+2\theta)\mu L_0^i}e^{\theta(W^i(0)+J_0^i(0))},1\},$$
with $\eta$ and $\theta$ defined in \eqref{eq: theta and eta}.
\end{proposition}

Our idea is to construct two dominant $M/GI/1$ queues starting from time $L_0^i$ and estimate the hitting times of $M/GI/1$ queues by a super-martingale.
Specifically, define the forward process  
$$\hat{R}(t,s)=\sum_{t\leq t_l^1<t+s}S_l-\mu s.$$
By the same argument of Proposition \ref{prop: mG1}, the forward process also has the dominance relationship by 
$$R(t,s)\leq J_0(t)+\hat{R}(t,s).$$ 
Therefore, when the $M/GI/1$ queue starting from the time $L_0^i$ clears, the Hawkes/GI/1 queue must have already cleared, i.e.,
$$\tau_i-L_0^i\leq \hat{\tau}_i\equiv \min\{t\geq 0:W^i(L_0^i)+J_0(L_0^i)+\hat{R}(L_0^i,t)=0\}.$$
To estimate $\hat{\tau}_i$, we define the following super-martingale 
$$M_i(t)=\exp(\theta (W^i(L_0^i)+J_0(L_0^i))+\theta\hat{R}(L_0^i,t)+\eta \mu t).$$
With the help of $M_i(t)$, we could give an exponential moment bound of $\hat{\tau}_i$, which gives the hitting time distribution via Markov inequality. The full proof of Proposition \ref{prop: mixing time} is in Section \ref{sec: proof ergo}. 

With Proposition \ref{prop: mixing time}, we are ready to complete the proof of	Theorem \ref{thm: geometric ergodicity}.
Recall that 
$$f\left(\vert W^1(t)-W^2(t)\vert\right)\leq  f\left(D_0\right)\ind{\max(\tau_1,\tau_2)>t},$$
so
$$\EE\left[f\left(\vert W^1(t)-W^2(t)\vert\right)\Big\vert\mathcal{F}_0\right]\leq  f\left(D_0\right)\PP(\max(\tau_1,\tau_2)>t\vert \mathcal{F}_0).$$
By Proposition \ref{prop: mixing time}, we can summarize that
\begin{align*}
&\EE\left[f\left(\vert W^1(t)-W^2(t)\vert\right)\Big\vert\mathcal{F}_0\right]\\
\leq&  f\left(D_0\right)\PP(\max(\tau_1,\tau_2)>t\vert \mathcal{F}_0)\leq f\left(D_0\right) \sum_{i=1}^2 \PP(\tau_i>t\vert \mathcal{F}_0)\\
\leq &f(D_0)\sum_{i=1}^{2}\min\{\exp\left((\eta+2\theta)\mu L_0^i+\theta(W^i(0)+J^i_0(0)-\eta \mu t)\right),1\}.
\end{align*}
Similarly, for the busy period, we notice that 
$$\vert X^1(t)-X^2(t)\vert \leq \left(X^1(0)+X^2(0)+t\right)\ind{\tau_1\vee\tau_2\geq t}.$$
Therefore, by Proposition \ref{prop: mixing time},
\begin{align*}
\EE\left[\vert X^1(t)-X^2(t)\Big\vert\mathcal{F}_0 \right]&\leq  \left(X^1(0)+X^2(0)+t\right)^m \PP(\tau_1\vee\tau_2\geq t\vert\mathcal{F}_0)\\
&\leq e^{-\eta\mu  t}\left(X^1(0)+X^2(0)+t\right)^m\sum_{i=1}^2\min\{e^{(\eta+2\theta)\mu L_0^i+\theta(W^i(0)+J_0^i(0))-\eta\mu t},1\}.
\end{align*} 



\subsection{Proof of Corollary \ref{coro: convergence to steady state}}\label{subsec: proof convergence to steady-state} 
To show the exponential ergodicity, we rely on our semi-synchronous coupling construction and Proposition \ref{prop: mixing time}. Let $W(\cdot),X(\cdot),N_0(\cdot)$ be the (transient) Hawkes$/GI/1$ with initial state $(W(0),X(0),N_0(0))$ and $(\tilde{W}(\cdot),\tilde{X}(\cdot),\tilde{N}_0(\cdot))$ be a stationary Hawkes$/GI/1$ queue semi-synchronously coupled with $(W(\cdot),X(\cdot),N_0(\cdot))$. Let $\mathcal{B}$ be the $\sigma$-field of $(W(t),X(t),N_0(t))$, then by definition,
\begin{align*}
&\|\PP^t(\cdot|W(0),X(0),N_0(0))-\pi\|_{TV}\\
=&\sup_{A\in\mathcal{B}} |\PP((W(t),X(t),N_0(t))\in A|W(0),X(0),N_0(0))-\PP((\tilde{W}(t),\tilde{X}(t),\tilde{N}_0(t))\in A)|\\
\leq & \PP(W(t)\neq \tilde{W}(t),X(t)\neq\tilde{X}(t),N_0(t)\neq \tilde{N}(t)|W(0),X(0),N_0(0))
\end{align*}
Note that by our construction, $(W(t),X(t),N_0(t))$ coincide with $(\tilde{W}(t),\tilde{X}(t),\tilde{N}(t))$ as long as $t>\max(\tau_1,\tau_2)$ (Lemma \ref{lmm: sychronous coupling2}). Consequently, we have 
\begin{align*}
&\|\PP^t(\cdot|W(0),X(0),N_0(0))-\pi\|_{TV}\\
\leq &\EE[\PP(\max(\tau_1,\tau_2)>t|\mathcal{F}_0)| W(0),X(0),N_0(0)]\\
\leq &\min\{e^{-\eta\mu t}(M_0+\exp((\eta+2\theta)\mu L_0+\theta (W(0)+J_0(0)))),1\}
\end{align*}
where $M_0=\EE[\exp((\eta+2\theta)\mu \tilde{L}_0+\theta(\tilde{W}(0)+\tilde{J}_0(0)))]<\infty$. The last inequality is by Proposition \ref{prop: mixing time}.
\section{Application Example:  Online Learning for Hawkes/GI/1 Queue}\label{sec: online learning}


In this section, we introduce an application of the theory developed in Section \ref{sec: main results}. Specifically, we introduce an online learning problem for Hawkes$/GI/1$ queues and design the first online learning algorithm for the staffing problem of Hawkes$/GI/1$ queues. In addition, we analyze the efficiency of this algorithm via the regret analysis and the convergence rate of the algorithm, the heart of which is the moment bounds and the ergodicity results in Section \ref{sec: main results}.
\subsection{Problem Setting and Assumptions}\label{subsec: learning setting}
We consider a Hawkes/GI/1 queue in which the customers arrive following a stationary Hawkes process with \textbf{unknown} parameter $(\lambda_0,m, f(\cdot))$ and the i.i.d. job size follows a positive random variable $V$ with mean $1$ with service rate $\mu$ as in Section \ref{sec: preliminaries}.  In the staffing problem, the service provider's goal is to find the best choice of $\mu\in\mathcal{B}= [\underline{\mu},\bar{\mu}]$ to balance between the staffing cost (per unit time) and the penalty of congestion: 
\begin{equation}\label{eq:objective}
\min_{\mu\in\mathcal{B}} f(\mu) \equiv h_0\mathbb{E}[W_\infty(\mu)]+c(\mu),
\end{equation} 
where $W_\infty(\mu)$ is the steady-state workload process under the staffing level $\mu$ and $c(\mu)$ is the corresponding staffing cost. For the ease of notation, let's denote $w(\mu)\triangleq \EE[W_\infty(\mu)]$. We shall additionally impose the following assumption for the rest of this article to guarantee the convexity of the objective, which is a common practice in online learning literature \citep{hazan2016introduction}. 
\begin{assumption}[Service Cost] We assume the staffing cost function $c(\mu)$ is convex, continuously differentiable and non-decreasing in $\mu$.
\label{assum: convex of cmu}
\end{assumption}
The form of the staffing problem \eqref{eq:objective} has a long history in traditional queueing systems (see, for example, \citep{Maglaras2003,LeeWard2014,Nair2016,onlinequeue2020} and references therein). However, in problem \eqref{eq:objective}, the workload $\EE[W_\infty(\mu)]$ has no closed form even for $GI/GI/1$ queue, let alone more complicated Hawkes$/GI/1$ queue, so the problem \eqref{eq:objective} is intractable analytically. Therefore, we design the first online learning algorithm for the Hawkes staffing problem \eqref{eq:objective} without knowledge of the information of Hawkes arrivals (to the best of our knowledge). 

\subsection{Online SGD-Based Algorithm and the System Dynamics}\label{subsec: online algorithm}

In this section, we introduce our algorithm designed for the Hawkes$/GI/1$ queue. Our algorithm is designed based on the framework of \textit{Gradient-based Online Learning in Queues} (GOLiQ) algorithm in our recent paper \cite{onlinequeue2020}, which is an algorithm specially designed for $GI/GI/1$ queue for a similar problem. Although we could not directly apply GOLiQ to the Hawkes system due to the difference between Hawkes queues and $GI/GI/1$ queue, our moment bound and ergodicity results in Section \ref{sec: main results} provide us the key ingredients to design and efficiency analysis of the new algorithm (Algorithm \ref{alg: direct}

\begin{algorithm}[H]
\caption{Online Algorithm for Hawkes$/GI/1$ Queues (GOLiQ-Hawkes)}\label{alg: direct}
\begin{algorithmic}[1]
	\Require {number of cycles $L$, 
		parameters $T_k$, $\eta_k$ for $k=1,2,...,L$,
		initial value $\mu_1$, warm-up rate $\xi$}
	\For{$n=1,2,...,L$}
	\State operate the system under $\mu_k$ until time $T_k$
	\State observe $X_k(t)$ for $t\in[0,T_k]$
	\State estimate gradient by 
	$$H_k=-\frac{h_0}{(1-\xi)T_k}\int_{\xi T_k}^{T_k}X_k(t)dt+c'(\mu)$$
	\State \textbf{update: } $\mu_{k+1}=\min\left(\max(\mu_k-\eta_kH_k,\underline{\mu}),\bar{\mu}\right)$
	\EndFor 
\end{algorithmic}
\end{algorithm}

Our algorithm falls into the category of gradient descent method. In our algorithm, we organize the time into successive cycles according to the hyperparameter $T_k$, the cycle length. To mitigate the bias introduced by the transient state of the queue, the algorithm sends $T_k$ to $\infty$ as $k\rightarrow\infty$. In each cycle, the algorithm operates the system under $\mu_k$ and builds a gradient estimator $H_k$ with data collected within the cycle. At the end of the cycle, the system updates $\mu_{k+1}$ by an SGD rule. Due to the cyclic design, we use the under-script $_k$ to denote the corresponding stochastic process in cycle $k$, e.g., $N_k(t),W_k(t),X_k(t),J_{0,k}(t)\cdots$ being the corresponding arrival, workload, busy time, total residual job sizes, etc., in cycle $k$. 

The heart of Algorithm \ref{alg: direct} is its gradient estimator. In our case, the challenge is to estimate $w'(\mu)$, which has no closed form, from data without involving any unknown Hawkes information, i.e., the parameters $(\lambda_0,m,f)$. For this purpose, we apply a pathwise analysis to the Hawkes system and derive a convenient expression of the derivative of $w(\mu)$ by the following representation theorem.
\begin{proposition}
Suppose Assumption \ref{assmpt: stable} holds. Then, for a $Hawkes/GI/1$ queue with stationary input, the derivative of the mean steady-state workload $w(\mu)$ satisfies 
$$w'(\mu)=-\mathbb{E}[X_\infty(\mu)],$$
where $X_\infty$ is the steady-state of the observed busy period (in continuous time) for $\mu\in\mathcal{B}$. Moreover, $w(\mu)$ is Lipchitz continuous for $\mu\in\mathcal{B}$, i.e.,
$$\vert w(\mu_1)-w(\mu_2)\vert \leq \EE\left[X_\infty(\underline{\mu})\right]\vert \mu_1-\mu_2\vert.$$
\label{prop: gradient design}
\end{proposition}
With this gradient representation, we apply a finite time sample average in each cycle to estimate the gradient in Algorithm \ref{alg: direct}. The key of Proposition \ref{prop: gradient design} is to analyze the gradient pathwisely. In detail, by the construction in Section \ref{sec: stationary haweks queue}, we could prove that almost surely,
$$\tilde{W}_\mu(0)'=\left(\max_{u\geq 0}R_\mu^\leftarrow(0,u)\right)'=-\tilde{X}_\mu(0).$$
Therefore, once we can interchange the derivative and the expectation, then Proposition \ref{prop: gradient design} is at hand and the interchange could be guaranteed by the moment condition in Theorem \ref{thm: moment of W}. For more details, see the full proof in Section \ref{subsec: proofs regret}.

\subsection{Regret Analysis}\label{subsec: regret analysis}
In this section, we show the efficiency of GOLiQ-Hawkes by providing corresponding performance bounds. Specifically, we consider the regret as the performance measure, which is commonly used in online learning literature, e.g., \cite{hazan2016introduction} and references therein. In a nutshell, regret is the difference between the performance of the system under the online learning algorithm and the performance under the optimal staffing level $\mu^*$. In detail, the regret accumulated in cycle $k$ is defined by
\begin{equation}
R_k = \underbrace{\EE\left[c(\mu_k)T_k + h_0\int_0^{T_k}W_k(t)dt\right]}_{\text{the cumulative cost in cycle $k$}}- \underbrace{f(\mu^*)T_k}_{\text{the optimal cost}}.
\end{equation}
As a result, the total regret in $L$ cycles is $R(L)=\sum_{k=1}^L R_k$.

Although we could not directly apply the analysis of GOLiQ to GOLiQ-Hawkes due to the differences between Hawkes arrivals and renewal arrivals, the analysis framework of GOLiQ in \cite{onlinequeue2020} sheds some light on the analysis of GOLiQ-Hawkes. Essentially, the success of GOLiQ is substantiated by (i) the boundedness of key queueing functions in each cycle; (ii) the geometric ergodicity of $GI/GI/1$ queue; and (iii) the convex structure of the revenue management objective. In other words, if we could provide the parallel conditions of (i) to (iii) ahead in the Hawkes version, GOLiQ-Hawkes would have an efficient convergence rate as GOLiQ does. Fortunately, our results in Section \ref{sec: main results} provide these properties adapted to Hawkes queue's version. We summarize these three properties as follows and note that these properties are direct results of our analysis in Section \ref{sec: main results}. The proofs of these properties are collected in Appendix Section \ref{subsec: Proofs of regret analysis}.

\begin{corollary}[Boundedness of Queueing Functions]\label{lmm:bounded}
Suppose Assumption \ref{assmpt: stable} to \ref{assume: light-tail of V} hold. For any control sequence $\mu_k\in\mathcal{B}$, with 
$$(\theta,\eta)=\frac{\min(\theta_0/2,\theta_1)}{6\max(\bar{\mu},1)(\eta_1+2\theta_1)}\cdot(\theta_1,\eta_1)$$
as defined in \eqref{eq: theta and eta}, the following queueing terms 
$$\EE[W_k(t)^{2m}],\EE[X_k(t)^{2m}], \EE[J_{0,k}(t)^{2m}],\EE[e^{6(\eta+2\theta)\mu_k L_{0,k}}],\EE[e^{6\theta W_k(t)}],\EE[e^{6\theta J_{0,k}(t)}]$$
are bounded by a constant $M$.
\end{corollary}

\begin{corollary}[Ergodicity of Queueing Functions in Each Cycle]\label{coro: ergodicity compact} Suppose that Assumption \ref{assmpt: stable} to \ref{assume: light-tail of V} hold. Let $\tilde{W}_k(t)$ be a stationary workload process semi-synchronously coupled with $W_k(t)$. In addition, let $X_k(t),\tilde{X}_k(t)$ be the corresponding busy time. Then, there exists a constant $A$ such that for $m=1,2$,
\begin{align*}
	\EE\left[\left\vert W_k(t)-\tilde{W}_k(t)\right\vert^m\right]&\leq Ae^{-\eta\underline{\mu} t}\\
	\EE\left[\left\vert X_k(t)-\tilde{X}_k(t)\right\vert^m\right]&\leq Ae^{-0.5\eta\underline{\mu} t},
\end{align*}
where $\eta$ is specified in \eqref{eq: theta and eta}.
\end{corollary}

\begin{corollary}[Convex and Smoothness]
\label{coro: convex and smooth}
Suppose Assumption \ref{assmpt: stable} to \ref{assum: convex of cmu} hold. The objective function $f(\mu)$ has a continuous first-order derivative denoted by $\nabla f(\mu)$. There exist finite positive constants $K_0\leq1$ and $K_1>K_0$ such that for all $\mu\in\mathcal{B}$
\begin{enumerate}
	\item $(\mu-\mu^*)\nabla f(\mu)\geq K_0(\mu-\mu^*)^2$
	\item $\vert \nabla f(\mu)\vert \leq K_1 \vert \mu-\mu^*\vert $
\end{enumerate}
\end{corollary}

Armed by the boundedness, exponential ergodicity and convex property, the efficiency of GOLiQ-Hawkes could be summarized by the following theorem. In short, we will show that with properly chosen hyper-parameters $\eta_k$ and $T_k$, GOLiQ-Hawkes is efficient in the sense that it has a logarithmic regret, which also indicates that the decision $\mu_k$ of GOLiQ-Hawkes will also converge to the clairvoyant policy $\mu^*$ in a rapid speed.

\begin{thm}\label{thm: main}
Under Assumption \ref{assmpt: stable} to \ref{assum: convex of cmu}, let $T_k=a_T+c_T\log(k)$, $\eta_k=c_\eta k^{-1}$ with any
$$c_T\geq 2(\eta \underline{\mu}\xi)^{-1}\max(\log(8A/K_0),1),\quad  \text{and}\quad c_\eta\geq 2k^{-1}/K_0,$$
and any $a_T\geq0$, where $\eta$ is specified in \eqref{eq: theta and eta}, $\xi$ is the warm-up rate in Alg.\ref{alg: direct}, $A$ is defined in Corollary \ref{coro: ergodicity compact}, and $K_0$ is the strong convexity constant in Corollary \ref{coro: convex and smooth}.
Then there exists a constant $C_{alg}>0$ such that, for all $L\geq 1$,
$$R(L)\leq C_{alg} \log(L)^2.$$
In particular, we have 
$$R(L) = O(\log(L)^2) = O(\log(M_L)^2),$$
with $M_L=\sum_{k=1}^L T_k$, i.e., the total units of time elapsed in $L$ cycles.
\end{thm}

\begin{remark}
Compared to conventional $\log(L)$ regret for online gradient descent \cite{hazan2016introduction}, the additional logarithm comes from the transient behavior of the Hawkes queues and the exponential convergence rate of Hawkes queues.
\end{remark}


With the three critical properties of Hawkes queues, the efficiency of GOLiQ-Hawkes could be provided following the analysis framework of \cite{onlinequeue2020}. Basically, the key idea of the analysis is to decompose the regret into the following two parts:
\begin{align*}
R_{1,k}&=\EE\left[h_0\int_{0}^{T_k}W_k(t)-w(\mu_k)dt\right]\\
R_{2,k}&=\EE\left[f(\mu_k)-f(\mu^*)\right]T_k.
\end{align*}
Then the total regret cumulated in $L$ cycles is decomposed accordingly as
\begin{equation*}
R(L)=\sum_{k=1}^L R_k=\sum_{k=1}^{L}R_{1,k}+\sum_{k=1}^L R_{2,k}\triangleq R_1(L)+R_2(L).
\end{equation*}
Intuitively, $R_1(L)$ is the regret caused by transient behavior of the queueing system (regret of the transient behaviors). On the other hand, $R_2(L)$ can be interpreted as the regret driven by suboptimality of the decision variable $\mu_k$ (regret of suboptimality). For the regret of transient behaviors $R_{1}$, we could apply the exponential ergodicity result. The regret of suboptimality $R_{2}$ hinges on the bias and variance of the gradient estimator $H_k$. We analyze the bias by ergodicity and the variance by boundedness. With the help of convexity, the regret of suboptimality $R_2$ can be well controlled, too. For the sake of brevity, we would like to omit the full proof here and defer the full proof in the  Appendix Section \ref{subsec: proofs regret}.

\section{Numerical Results}\label{sec: num}
In this section, we first confirm the effectiveness of GOLiQ-Hawkes via numerical experiments. Then we investigate the impact of autocorrelation of the Hawkes process on the staffing decision by comparing it with a $GI/GI/1$ queue with same marginal inter-arrival time distribution. Finally, we study the optimal staffing in a heavy-traffic regime to see  whether the classic ``square-root rule" still holds for Hawkes queues. In all numerical experiments, the stationary Hawkes arrivals is generated using the perfect sampling algorithm 1 developed in \cite{perfect2020}.
\subsection{Performance Evaluation of GOLiQ-Hawkes}
\label{subsec: base examples}
We consider problem \eqref{eq:objective} with a quadratic staffing cost $c(\mu)=c_0\mu^2$ in this section. In specific, we set $h_0=c_0=0.5$. As a result, the objective \eqref{eq:objective} reduces to
\begin{equation}\label{CostNumerical}
\min_{\mu\in\mathcal{B}}f(\mu)=\min_{\mu\in\mathcal{B}}\left\lbrace 0.5\mu^2+0.5\mathbb{E}\left[W_{\infty}(\mu)\right] \right\rbrace.
\end{equation}
Since there is no closed form of $\EE[W_\infty(\mu)]$, the exact solution $\mu^*$ of problem \eqref{CostNumerical} is unavailable. We apply the \textit{naive grid search method} (NGS) and set the solution of NGS as the benchmark for regret estimation. In detail, for each candidate $\mu$, we operate Hawkes$/GI/1$ queues under $\mu$ for $T=100,000$ unit times and estimate $\EE[W_\infty(\mu)]$ by time average estimator $\frac{1}{0.5T}\int_{0.5T}^{T}W(t)dt$. The estimate $\hat{f}(\mu)$ is achieved based on $100$ independent time averaging estimations above. By comparing all $\hat{f}(\mu)$, we choose $\mu$ with the smallest $\hat{f}(\mu)$ as yardstick $\mu^*$, with which we are able to benchmark the solutions of GOLiQ-Hawkes.

We consider non-Markovian Hawkes arrivals with unit background intensity $\lambda_0=1$ and Gamma kernel $m(t)=2te^{-2t}$. In this case, $m=\int_{0}^{\infty}m(t)dt=1/2$, and the average arrival rate $\lambda=2$. For service times, we consider Erlang-2 distribution with a rate of $2\mu$. The NGS method gives out the benchmark that $\hat{\mu}^*=2.84$ and $\hat{f}(\mu^*)=6.11$ (with a precision of 0.01), and the cost function is shown in the top left panel of Figure \ref{fig: HawkesGI1} as well. By Algorithm \ref{alg: direct} and Theorem \ref{thm: main}, we select step length $\eta_k=1/k$ and cycle length $T_k=10+20\log(k)$ with initial point $\mu_0=10$. In Figure \ref{fig: HawkesGI1}, we again give a sample path of staffing level $\mu_k$ and the estimated regret curve (obtained by 1,000 independent Monte-Carlos runs). We find that the staffing level $\mu_k$ converges to the optimal value $\mu^*$ rapidly, and the regret grows as a logarithmic function of time $t$. Particularly, a simple linear regression with respect to $\left(\sqrt{R(t)},\log(t)\right)$ (bottom right panel) validates our regret bound in Theorem \ref{thm: main}.

\begin{figure}

\centering
\includegraphics[width=0.8\linewidth]{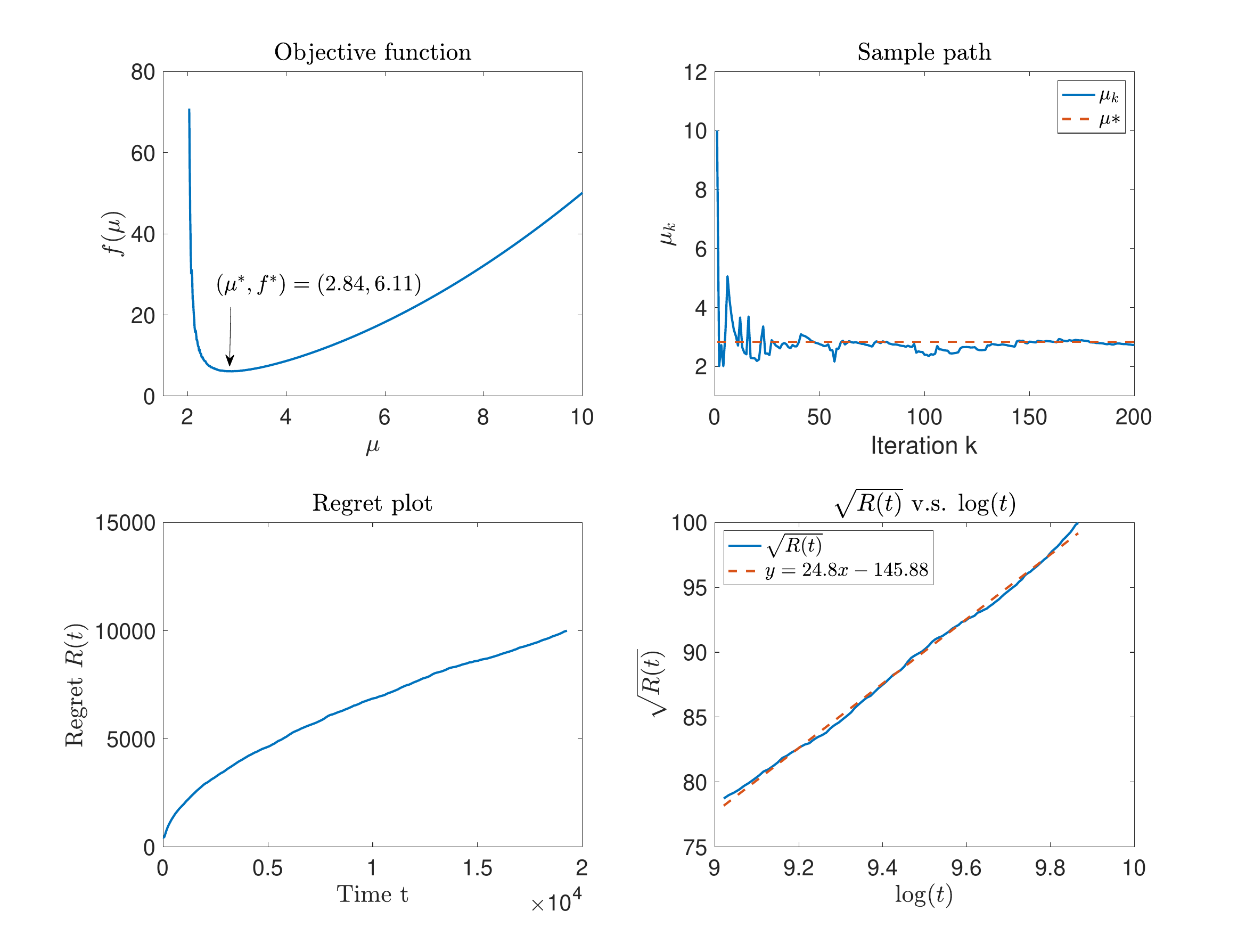}
\caption{Online optimal staffing for non-Markovian Hawkes$/GI/1$ by GOLiQ-Hawkes with $\mu^*=2.84$, $\eta_k=1/k$ and $T_k=10+20\log(k)$: (i) objective function (top left); (ii) sample path of the staffing level $\mu_k$ (top right); (iii) estimated regret (bottom left); (iv) square root of regret versus logarithmic of the total time elapsed (bottom right).}
\label{fig: HawkesGI1}
\end{figure}

We further test the robustness of GOLiQ-Hawkes by applying different hyperparameters $\eta_k$ and $D_k$. Specifically, we first validate the robustness of $\eta_k$ by choosing $\eta_k=c_\eta k^{-1}$ with $c_\eta\in\{0.5,1,3,5\}$ and fixing $T_k=10+20\log k$. Next, we verify the robustness of $T_k$ by selecting $T_k=10+c_T\log k$ with $c_T\in\{10,20,30\}$ and fixing $\eta_k=k^{-1}$. For each case, the regret curve is estimated by 1,000 independent runs, and in Figure \ref{fig: robustness}, we also draw $\sqrt{R(t)}$ v.s. $\log(t)$ plot. Figure \ref{fig: robustness} reveals that GOLiQ-Hawkes continues to perform efficiently with different hyperparameters with all regret curves exhibiting a logarithmic order. In summary, our numerical experiments illustrate that GOLiQ-Hawkes performs effectively with a logarithmic regret and that GOLiQ-Hawkes is robust to the hyperparameter choices.

\begin{figure}
\centering 
\includegraphics[width=\linewidth]{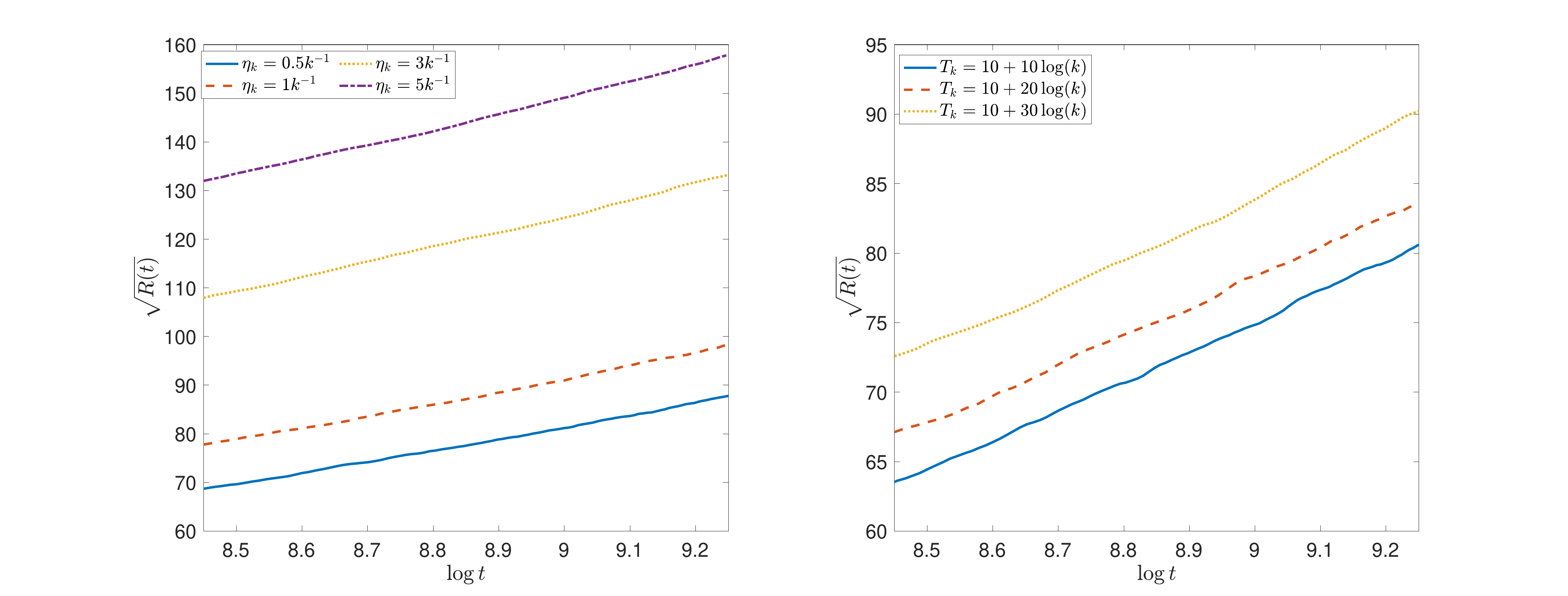}
\caption{Simulated $\log t$ v.s. $\sqrt{R(t)}$ plot with different hyperparameters (i) step size $\eta_k$ robustness (left panel); (ii) cycle length $T_k$ (right panel). All regret curves are estimated by averaging 1,000 independent replications.}
\label{fig: robustness}
\end{figure}

\subsection{Impact of Self-excitement of Hawkes arrivals}
\label{subsec: self-excitement}
In this section, we conduct numerical experiments to investigate the impact of the self-excitement of Hawkes arrivals on the optimal stuffing level $\mu^*$. Specifically, we compare the optimal staffing level between queueing systems with Hawkes arrivals and systems with renewal arrivals, of which the interarrival times have the same distribution as the marginal distribution of the stationary Hawkes process' interarrival times.

For this purpose, we generate the arrival process $(GI)$ by the following shuffling method: (i) we first generate a sufficiently long stationary Hawkes process's path; (ii) all of the interarrival times of this Hawkes path are randomly shuffled; (iii) the ideal $GI$ renewal process is generated by splicing these shuffled interarrival times. Then GOLiQ-Hawkes is applied to both the Hawkes$/GI/1$ queue and the corresponding $GI/GI/1$ queue to obtain the optimal level $\mu_{H}^*$(the optimal staffing for Hawkes queue) and $\mu_{GI}^*$ (the optimal level for $GI/GI/1$ queue).

Since we intend to investigate the impact of the autocorrelation on optimal staffing, we compare $\mu_{H}^*$ with $\mu_{GI}^*$ under different autocorrelation levels for objective \eqref{CostNumerical}. A convenient indicator of autocorrelation level for Hawkes process is $m=\int_{0}^{\infty}m(t)dt$, i.e., the average number of offspring of a single arrival. Therefore, in this experiment, we consider kernels in the form:
$$m(t)=a\cdot e^{-2t},\quad a\in\{0,0.2,\cdots,1.8\},$$
so that $m=\{0,0.1,\cdots,0.9\}$. As higher autocorrelation level results in a larger arrival rate, to make experiments comparable, the background intensity $\lambda_0$ is chosen as $\lambda_0=2\cdot(1-m)$ accordingly so that the total arrival rate $\lambda=\frac{\lambda_0}{1-m}=2$ across all the cases. In this way, we have a sequence of systems whose total arrival rates are the same but the auto-correlation level increases with $m$. As for the service distribution, we consider exponential service times for simplicity. Then, we investigate the impact of the auto-correlation level of the Hawkes process by comparing the difference of $\mu^*_{GI}$ and $\mu^*_{H}$ under different auto-correlation level $m$ in each case.

For each case, we apply the GOLiQ-Hawkes for $L=300$ iterations for $100$ independent replications. The  $\mu_{\text{Hawkes}}^*$ and $\mu_{\text{GI}}^*$ is estimated by averaging the staffing level in the last iteration, i.e., $\mu_{L}$ of the $100$ independent paths. The hyperparameters for GOLiQ-Hawkes in both Hawkes queue and $GI/M/1$ queue are $\eta_k=3/k,T_k=10+20\log(k)$. The results are reported in Table \ref{tab: compare Hawkes}.

\begin{table}
	\centering
	\begin{tabular}{cccccc}
		\hline
		\multirow{2}{*}{m} & \multicolumn{2}{c}{Hawkes/M/1} & \multicolumn{2}{c}{GI/M/1} & \multirow{2}{*}{\begin{tabular}[c]{@{}c@{}}Relative Increase \\ in $\mu^*$\end{tabular}} \\ \cline{2-5}
		& $\mu_{H}^*$   & Variance   & $\mu_{GI}^*$ & Variance &                             
		\\ \midrule
		0            & 2.61                                                       & 0.0014             & 2.61                & 0.0016             & 0.00\%                                                                             \\
		0.1          & 2.64                                                       & 0.0019             & 2.62                & 0.0015             & 0.64\%                                                                             \\
		0.2          & 2.68                                                       & 0.0019             & 2.65                & 0.0014             & 1.13\%                                                                             \\
		0.3          & 2.72                                                       & 0.0017             & 2.66                & 0.0016             & 2.24\%                                                                             \\
		0.4          & 2.80                                                       & 0.0028             & 2.70                & 0.0021             & 3.88\%                                                                             \\
		0.5          & 2.90                                                       & 0.0027             & 2.74                & 0.0013             & 5.56\%                                                                             \\
		0.6          & 3.05                                                       & 0.0039             & 2.79                & 0.0026             & 9.12\%                                                                             \\
		0.7          & 3.27                                                       & 0.0066             & 2.87                & 0.0016             & 14.05\%                                                                            \\
		0.8          & 3.70                                                       & 0.0130             & 3.03                & 0.0034             & 22.14\%                                                                            \\
		0.9          & 4.62                                                       & 0.0583             & 3.40                & 0.0052             & 35.97\%                                                                            \\ \bottomrule
	\end{tabular}
	\caption{The estimated value and variance of optimal service capacity solved by GOLiQ-Hawkes for 100 replications for Hawkes$/M/1$ and the corresponding $GI/M/1$ with different levels of self-excitement $m$.}\label{tab: compare Hawkes}
\end{table}

Table \ref{tab: compare Hawkes} illustrates that as the autocorrelation level $m$ grows, both Hawkes$/M/1$ and $GI/M/1$ queues need higher staffing levels to accommodate the surge of arrivals. However, the Hawkes$/M/1$ queue needs significantly more service capacity than the $GI/M/1$ queue due to Hawkes' autocorrelation at all levels of $m$. In addition, the relative difference in staffing levels, i.e.,
$$\frac{\mu_{H}^*-\mu_{GI}^*}{\mu_{GI}^*}$$
grows with $m$, which implies that the difference between Hawkes queues and $GI/GI/1$ queues inclines to be sharper as the autocorrelation level increases, and consequently, the Hawkes$/GI/1$ queue needs significantly more service capacity as Hawkes' autocorrelation level increases.

\subsection{Does ``square root rule" hold for Hawkes queues?}
\label{subsec: sqrt rule}
The square root rule is prevalent in many queueing models for determining the optimal staffing level. Specifically, \cite{LeeWard2014} proved asymptotic optimality of the square root rule for $GI/GI/1$ queue with linear staffing cost, i.e. the optimal staffing level satisfies
$$\mu^*=\lambda+O(\sqrt{\lambda}),\quad \text{as }\lambda\rightarrow\infty.$$

To check whether the square root rule still holds for Hawkes$/GI/1$ queue, we design a set of numerical experiments as follows. We consider the staffing problem with linear cost as
$$\min_{\mu} 2\mu+0.1\EE[W_\infty(\mu)]$$
on a set of Hawkes$/GI/1$ queues with same exponential service time distribution but with increasing avaerage intensity $\lambda$ in the Hawkes arrival processes. We then observe how the optimal staffing level $\mu^*_{H}$ changes with $\lambda$ and whether it increases following the square root rule. As a comparison, we also solve the same staffing problem for a $GI/GI/1$ queue with IID inter-arrival times follow the same marginal distribution as that of Hawkes queue constructed in the same way as described in Section \ref{subsec: self-excitement}.

As the major difference lies in autocorrelation, we send the arrival rate $\lambda$ to infinity by letting $m\rightarrow1$. In specific, we consider a sequence of Hawkes queueing systems indexed by $n$ with the same background intensity $\lambda_0=1$. However, the kernel of the Hawkes process in each system changes with $n$ in the form of  $m(t)=(2-2/n)e^{-2t}$ so that the total arrival rate 
$$\lambda=\frac{\lambda_0}{1-m}=n$$
in the $n$-th system. In the experiment, we consider $n\in\{5,25,45,65,85,105\}$. We investigate how $\mu^*_{H}$ and $\mu^*_{GI}$ change with $n$ and whether they obey the square root rule.

To derive the solutions $\mu^*_{H}$ and $\mu^*_{GI}$, in each case, we apply GOLiQ-Hawkes to both Hawkes$/GI/1$ and $GI/GI/1$ systems by 100 replications, and the optimal levels are estimated in the same procedure in Section \ref{subsec: self-excitement} with the same hyperparameters. If the square root rule still holds for Hawkes$/GI/1$ queue, then $\mu_H^*-\lambda$ shall be linear with  $\sqrt{\lambda}$ and the ratio between them shall be close to a constant. The results are reported in Table \ref{tab: sqrt rule}.


In Table \ref{tab: sqrt rule}, we report the total arrival rate, the average number of descendants $m$, the \textit{auto-correlation function} (ACF) with lag-1 and the estimated $\mu^*_H$ (and $\mu^*_{GI}$). From Table \ref{tab: sqrt rule}, we can observe that as $m$ increases, $\lambda$ and ACF increase at the same time, which results in the increase of both staffing levels $\mu_{H}^*$ and $\mu_{GI}^*$. Parallel to the experiment in Section \ref{subsec: self-excitement}, the need for staffing in Hawkes queues $\mu^*_H$ is always larger than the staffing in renewal case $\mu^*_{GI}$. In addition, the difference between Hawkes queue and $GI/GI/1$ queue is quite sharp because a mild autocorrelation (ACF$=0.401$) may lead to $50\%$ more staffing! And the difference seems to enlarge with the increase of $\lambda$. 

In addition, we can observe in Table \ref{tab: sqrt rule} that, as $\lambda$ increases, the ratio $\frac{\mu^*_H-\lambda}{\sqrt{\lambda}}$ increases dramatically in Hawkes queue, while for $GI/GI/1$ queue, the ratio $\frac{\mu_{GI}^*-\lambda}{\sqrt{\lambda}}$ is around the same level of 0.35. In other words, this result confirms that the square-root rule works for $GI/GI/1$ queue but not for the Hawkes$/GI/1$ queue especially when the increasing of average arrival rate is driven by the self-excitement effect, i.e. $m\to 1$. We leave theoretic justification on this result for further study.

\begin{table}[]
	\centering
	\begin{tabular}{lllllllll}
		\multicolumn{1}{c}{\multirow{2}{*}{$\lambda$}} & \multicolumn{1}{c}{\multirow{2}{*}{$m$}} & \multicolumn{1}{c}{\multirow{2}{*}{ACF}} & \multicolumn{3}{c}{Hawkes/$GI/1$}                                   & \multicolumn{3}{c}{$GI/GI/1$}                                              \\
		\cline{4-9}
		\multicolumn{1}{c}{}                           & \multicolumn{1}{c}{}   &\multicolumn{1}{c}{}                    & $\mu^*_{H}$ & $\rho^*_H$ & $\frac{\mu^*_H-\lambda}{\sqrt{\lambda}}$ & $\mu^*_{GI}$ & $\rho^*_{GI}$ & $\frac{\mu^*_{GI}-\lambda}{\sqrt{\lambda}}$ \\
		\hline
		5                                              & 0.800                                    & 0.235& 6.16        & 0.812      & 0.52                                     & 5.78         & 0.865         & 0.35                                        \\
		25                                             & 0.960                                    & 0.331& 33.6        & 0.744      & 1.72                                     & 26.7         & 0.936         & 0.34                                        \\
		45                                             & 0.978                                    & 0.349& 63.4        & 0.709      & 2.74                                     & 47.2         & 0.953         & 0.33                                        \\
		65                                             & 0.985                                    & 0.354& 95.7        & 0.680      & 3.81                                     & 67.8         & 0.959         & 0.35                                        \\
		85                                             & 0.988                                    & 0.362& 131.2       & 0.648      & 5.01                                     & 88.4         & 0.962         & 0.37                                        \\
		105                                            & 0.991                                    & 0.401& 166.4       & 0.631      & 5.99                                     & 109        & 0.967         & 0.35                                       
	\end{tabular}
	\caption{The estimated $\mu^*$ for Hawkes$/M/1$ queues and $GI/M/1$ queues with different $\lambda$.}
	\label{tab: sqrt rule}
\end{table}


\section{Conclusion}\label{sec: conclusion}
In this paper, we establish moment bounds and exponential ergodicity for the general Hawkes$/GI/1$ queues under the stability condition. The key challenge in our analysis lies in the fact that the initial state of the Hawkes system contains the information of future arrivals and hence, the conventional approach is not applicable. To remedy this issue, we develop a semi-synchronous coupling method decomposing the future arrivals into two parts and we deal with them in different manners, which, we believe, is of independent research interest and can be useful for other stochastic models related to Hawkes processes. As an application of the theoretic results, we develop an efficient numerical tool for solving the optimal staffing problems in Hawkes$/GI/1$ queue. Our numerical results document the significant impact of clustering or auto-correlation in the arrival processes on decision-making of staffing service systems.  

There are several venues for future research. Since we show numerically that the square-root rule is not always satisfactory for Hawkes queues, one natural direction is to theoretically verify the numerical finding and identify the correct asymptotic staffing rule. Another interesting dimension is to investigate Hakwes queues with long-tail excitation function and possibly heavy tailed service times.
\newpage
%
%
%






\bibliographystyle{chicago}
\bibliography{reference}
\newpage

\begin{center}
	{\large\bf APPENDIX}
\end{center}
	\section{Proofs}
	In this section, we give the proofs of auxiliary results in the main paper.
	\subsection{Proof of Auxiliary Results in Section \ref{sec: preliminaries}}
	\begin{proof}{Proof of Lemma \ref{lmm: light-tail workload assumption}}
		Since $0<m<1$ as assumed, we have $m-1-\log(m)>0$. As a result, there exists $\theta_0$ such that $0<\log \EE[\exp(\theta_0 b)]<m-1-\log(m)$. Note that if moment generating function exists around $0$, then it's continuous differentiable in its domain \cite[p.78]{MGFbook}.
		Therefore, for \eqref{eq: theta1 def}, since $\psi_{V}(\cdot)$ is smooth around $0$, we have $\psi_{S_1}(\theta)=\psi_{\vert C_1\vert}(\log \psi_V(\theta))$ is also smooth around $0$. In this case, we can find a $\theta_1$ satisfying \eqref{eq: theta1 def} as follows. Since $\psi_V(\cdot)$ is smooth around $0$, then there exists a $\theta_0'$ such that $\psi_V(\theta_0')<\psi_b(\theta_0/2)$. For the other condition, define an auxiliary function 
		$$h(\theta)=\quad\frac{\lambda_0}{\underline{\mu}-{\varepsilon}}(\psi_{S_1}(\theta)-1)-\theta.$$
		By direct calculation, $h(0)=0,h'(0)=\frac{\lambda_0}{(1-m)(\underline{\mu}-{\varepsilon})}<0$ and $h''(\theta)=\frac{\lambda_0}{\underline{\mu}-{\varepsilon}}\psi_{S_1}''(\theta)=\frac{\lambda_0}{\underline{\mu}-{\varepsilon}}\EE[S_1^2e^{\theta S_1}]>\frac{\lambda_0}{\underline{\mu}-{\varepsilon}}\EE[S^2]>0$ for any $\theta>0$. Therefore, there exists one and only one positive solution for $h(\bar{\theta})=0$, and  denote it by $\bar{\theta}$. Clearly, $h(\theta)<0$ for $\theta\in(0,\bar{\theta})$. Then, we can choose 
		$$\theta_1=\min(\theta_0',\bar{\theta}/2),$$
		which satisfies \eqref{eq: theta1 def}. 
	\end{proof}
	\subsection{Proofs of the Auxiliary Result in Section \ref{subsec: CFTP construction} and \ref{subsec: mG1} }
	\begin{proof}{Proof of Proposition \ref{prop: stationary workload}}Under Assumption \ref{assmpt: stable}, for any $t\geq 0$, the process $R^\leftarrow(t,s)\to-\infty$ almost surely as $s\to\infty$. As a consequence, $\tilde{W}(t)$ is well-defined. In addition, for any $t$, the process $R^{\leftarrow}(t,\cdot)$ follow the same distribution as $R^{\leftarrow}(0,\cdot)$. So  $\tilde{W}(t)$ has the same distribution as $\tilde{W}(0)$ for all $t\geq 0$, and hence is stationary. 
		
		Now we show that $\tilde{W}(\cdot)$ follows the dynamic of a workload process of a Hawkes/GI/1 queue. Note that  $\tilde{W}(0)$ is the initial workload, customers arrive according to $N(t)$ with i.i.d. job size $V_n$, and service rate is $\mu$. For any $t\geq 0$ and $s>t$, 
		\begin{equation*}
			\begin{aligned}
				R^{\leftarrow}(t, s) &= 
				\sum_{n= N(t-s)}^{N(t)} V_n - \mu s\\
				&= \sum_{n= N(t-s)}^{-1} V_n- \mu (s-t) +\sum_{n=1}^{N(t)}V_n -\mu  t\\
				&= R^{\leftarrow}(0,s-t) +R(t).
			\end{aligned}
		\end{equation*}
		As a consequence, 
		$$\max_{s\geq t}R^{\leftarrow}(t, s) = \max_{s\geq 0} R^{\leftarrow}(0,s) + R(t) = \tilde{W}(0) + R(t).$$
		Recall that $R(t)$ is defined in \eqref{eq: R}. For $0\leq s\leq t$, we have
		$$R^{\leftarrow}(t,s) = R(t) - R(0,t-s).$$
		Then,
		\begin{equation*}
			\begin{aligned}
				\tilde{W}(t) &= \max_{s\geq 0}R^{\leftarrow}(t,s) = \left(\tilde{W}(0) + R(t)\right)\vee \max_{0\leq s\leq t} (R(t) - R(s))\\
				&= R(t)  - (-\tilde{W}(0))\wedge \min_{0\leq s\leq t}(R(s))\\
				&= \tilde{W}(0) + R(t)  -0\wedge \min_{0\leq s\leq t} \left\{\tilde{W}(0)+R(s)\right\}.
			\end{aligned}
		\end{equation*}
		According to \eqref{eq: workload dynamics}, we can conclude that $\tilde{W}(t)$ has the same dynamics as the of a Hawkes/GI/1 queue in which customers arrive according to the stationary Hawkes process $N(t)$ with i.i.d. job sizes $V_n$ and are served with rate $\mu$. 
	\end{proof}
	
	\begin{proof}{Proof of Proposition \ref{lmm: stationary X}}
		For convenience, define $t_0=t-\tilde{X}(t)$. Then, 
		\begin{align*}
			\tilde{W}(t_0)&=\max_{s\geq 0} R^\leftarrow(t_0,s)\\
			&=\max_{s\geq 0} R^\leftarrow(t,t-t_0+s)-R^\leftarrow(t,t-t_0)\\
			&=\max_{s\geq 0} R^\leftarrow(t,t-t_0+s)-R^\leftarrow(t,\tilde{X}(t))\leq  0.
		\end{align*}
		Since $\tilde{W}(t)\geq 0$, we have $\tilde{W}(t_0)=0$.
		Similarly, if there is a $t_1$ such that $t_0<t_1\leq t$, and $\tilde{W}(t_1)=0$, then 
		$$R^\leftarrow(t,t-t_1)\geq R^\leftarrow(t,t-t_0),$$
		which is impossible by definition of $\tilde{X}(t)$ and $t_0$. So $t_0=\sup\{s\leq t:\tilde{W}(s)=0\}$. This finishes the proof. 
	\end{proof}
	
	\begin{proof}{Proof of Proposition \ref{prop: mG1}}
		By definition, 
		\begin{align*}
			R^\leftarrow(t,s)&=
			\sum_{n= N(t-s)}^{N(t)} V_n - \mu s\\
			&=\sum_{t-s\leq \delta_l<t} \sum_{k=1}^{\vert C_l\vert}V_l^k\ind{t_l^k\geq t-s}+\sum_{t_l^1<t\leq \delta_l}\sum_{k=1}^{\vert C_l\vert}V_l^k \ind{t_l^k\leq t}-\mu s\\
			&\leq \sum_{t-s\leq \delta_l<t} \sum_{k=1}^{\vert C_l\vert}V_l^k-\mu s+\sum_{ t_l^1<t\leq \delta_l}\sum_{k=1}^{\vert C_l\vert}V_l^k \ind{t_l^k\leq t}\\
			&=\hat{R}^\leftarrow(t,s)+J_0^\leftarrow(t).
		\end{align*}
		Taking the maximum over $s\geq 0$ on both sides, we have the result of (a).
		
		For statement $(b)$, note that $\hat{W}(t)$ only depends on clusters that have departed before time $t$, i.e., $\delta_l<t$, while $J_0^{\leftarrow}(t)$ only involves clusters that will depart after time $t$. Given that clusters are independent of each other, we can conclude that $\hat{W}(t)$ and $J_0^\leftarrow(t)$ are independent of each other for any given $t$.
		
	\end{proof}
	\subsection{Proofs of Auxiliary Result in Section \ref{subsec: workload moments}}
	\begin{proof}{Proof of Proposition \ref{prop: J0}}
		By definition, $J_0(0)$ and $J_0^\leftarrow(0)$ are bounded by the sum of job sizes of all customers in $N_0(0)$.  For each cluster $C_l$, we define its total birth time as
		$$TB_l=\sum_{k=1}^{\vert C_l\vert}b_l^k.$$ 
		In addition, define two sets
		\begin{align*}
			C_{B_0}&=\left\{C_l:t_l^1<0,TB_l\geq-t_l^1\right\},\\
			B_0&=\left\{(k,t_l^k,p_l^k,V_l^k): (k,t_l^k,p_l^k,V_l^k)\in C_l, s.t.~ t_l^1<0, TB_l\geq -t_l^1\right\},
		\end{align*}
		i.e., a set of clusters whose total birth time is larger than its age at time 0 and a set of events from these clusters.  As the total birth time of a cluster is larger than the cluster length $\delta_l-t_l^1$, $N_0(0)$ must be a subset of $B_0$. 
		Therefore, we have 
		$$J_0(0),J_0^\leftarrow(0)\leq \sum_{(k,t_l^k,p_l^k,V_l^k)\in N_0(0)}V_l^k\leq  \sum_{(k,t_l^k,p_l^k,V_l^k)\in B_0}V_l^k\equiv \bar{J}_0.$$
		So, to prove that statements (a) and (b) hold for $J_0(0)$ and $J_0^\leftarrow(0)$, it suffices to show that the two statements are true for $\bar{J}_0$.
		
		Let $N_{B_0}\equiv\vert C_{B_0}\vert$ be the number of clusters whose total birth time is larger than its age at time 0.  
		Note that $TB_l$ are i.i.d. By Poisson thinning theorem, $N_{B_0}$ is a Poisson random variable with mean
		$$\EE[N_{B_0}]=\int_{-\infty}^{0}\lambda_0 \PP(TB>-t)dt=\lambda_0\EE[TB].$$
		In the following proof, we just index the clusters in $C_{B_0}$ as $1,2,...,N_{B_0}$ for the simplicity of notation.
		
		Next, we shall prove the two statements for $\bar{J}_0$ one by one.
		
		\begin{enumerate}
			\item[(a)] For any $n\geq 1$, 
			\begin{align*}
				\EE\left[\bar{J}_0^n\right]&=\EE\left[\left(\sum_{l=1}^{N_{B_0}}S_l\right)^n\right]
				\leq \EE\left[N_{B_0}^{n-1}\sum_{l=1}^{N_{B_0}}S_l^n\right]\\
				&=\EE\left[N_{B_0}^n\right]\EE\left[S_l^n\mid  C_l\in C_{B_0}\right],
			\end{align*}
			where the last equality holds as the clusters in $C_{B_0}$ are i.i.d conditional on $N_{B_0}$ by Poisson thinning. Moreover, conditional on $\vert C_l\vert$, the event $\{TB_l>-t_l^1\}$ is independent of the job sizes $V_l^k$ for $k=1,2,...,\vert C_l\vert$. Then we have
			\begin{align*}
				\EE\left[S_l^n\mid  C_l\in C_{B_0}\right]&=\EE\left[\left(\sum_{k=1}^{\vert C_l\vert}V_l^k\right)^n\Big\vert- t_l^1<TB_l\right]
				\leq \EE[V_1^n]\EE\left[\vert C_l\vert^n\mid -t_l^1<TB_l\right].
			\end{align*}
			By Poisson thinning, conditional on $C_l\in C_{B_0}$, the arrival times of each cluster $t_l^1$ has probability density function $p(-t)=\PP(TB_l>t)/\EE[TB_l]$ for $t\geq 0$. So, 
			\begin{align*}
				&\EE\left[\vert C_l\vert^n\mid -t_l^1<TB_l\right]=\int_{0}^\infty\EE[\vert C_l\vert^n\mid TB_l>t]p(-t)dt\\
				\leq & \int_{1}^{\infty}\EE[\vert C_l\vert^n TB^{n+1}\vert TB>t]t^{-n-1}p(-t)dt
				+\int_{0}^{1}\EE[\vert C_l\vert ^n]\PP(TB>t)^{-1}p(-t)dt\\
				\leq &~\EE[\vert C_l\vert^n TB^{n+1}]\int_{1}^{\infty}t^{-n-1}\PP(TB>t)^{-1}p(-t)dt
				+\EE[\vert C_l\vert ^n]/\EE[TB]\\
				=&\left(n^{-1}\EE[\vert C_l\vert^n TB_l^{n+1}]+\EE[\vert C_l\vert^n ]\right)/\EE[TB]\\
				=&\left(n^{-1}\EE[\vert C_l\vert^{2n}]^{1/2}\EE[ TB^{2n+2}]^{1/2}+\EE[\vert C_l\vert^n ]\right)/\EE[TB].
			\end{align*} 
			Note that $\vert C_l\vert$ is the total number of arrivals in a cluster generated by a Poisson branching process with $m<1$, it is known that $\vert C_l\vert$ follows Borel distribution (see, for example, \cite{dwass1969total}). Therefore, $\vert C_l\vert$ has finite $n$-th moments for all $n\geq 1$. On the other hand, $TB_l$ is a compound random variable and equals the sum of $\vert C_l\vert$ i.i.d. copies of birth time $b$.  By Assumption \ref{assmpt: light tail b}, $TB$ also has  finite $n$-th moment for all $n\geq 1$  (see Prop.3 of \cite{perfect2020}). So $\EE\left[\vert C_l\vert^n\mid -t_l^1<TB_l\right]<\infty$  and we can conclude that
			$$\EE[\bar{J}_0^n]\leq \EE[N_{B_0}^n]\EE[V^n]\EE[\vert C_l\vert^n \mid -t_l^1<TB_l]<\infty.$$
			
			\item[(b)] By Lemma \ref{lmm: light-tail workload assumption},
			$\psi_{V}(\theta_1)<\psi_b(\theta_0/2),$ then
			\begin{align*}
				\EE\left[\exp(\theta_1 \bar{J}_0)\right]&= \EE\left[\exp\left(\theta_1 \sum_{l=1}^{N_{B_0}}S_l\right)\right]
				=\EE\left[\EE\left[\exp\left(\theta_1 \sum_{k=1}^{\vert C_l\vert}V_l^k\right)\Bigg\vert TB_l>-t_l^1 \right]^{ N_{B_0}}\right]\\
				&=\EE\left[\EE\left[\exp\left(\vert C_l\vert \log (\psi_V(\theta_1))\right)\Bigg\vert TB_l>-t_l^1 \right]^{ N_{B_0}}\right]\\
				&\leq \EE\left[\EE\left[\exp\left(\vert C_l\vert \log (\psi_b(\theta_0/2))\right)\Bigg\vert TB_l>-t_l^1 \right]^{ N_{B_0}}\right].
			\end{align*}
			Note that conditional on $TB_l>-t_l^1$, the birth times $b_i$ tend to be larger compared to the unconditional distribution of $b_i$ intuitively. Therefore, we next prove that 
			$$\EE\left[\exp\left(\vert C_l\vert \log (\psi_b(\theta_0/2))\right)\Bigg\vert TB_l>-t_l^1 \right]\leq\EE\left[\exp\left(\frac{\theta_0}{2}TB_l\right)\Bigg\vert TB_l>-t_l^1\right],$$  
			and then give a bound on $\EE\left[\exp\left(\frac{\theta_0}{2}TB_l\right)\mid TB_l>-t_l^1\right]$ using Markov inequality. We show this claim by a stochastic ordering approach. By direct calculation, for any $t,s>0$ and $n>0$,
			\begin{align*}
				\PP(TB_l>s\big\vert \vert C_l\vert=n)&\leq \PP(TB_l>s\big\vert \vert C_l\vert=n, TB_l>t)\\
				&=
				\begin{cases}
					1,\quad \text{if }s<t,\\
					\frac{\PP(TB_l>s\big\vert \vert C_l\vert=n)}{\PP(TB_l>t\big\vert \vert C_l\vert=n)},\quad \text{if }s\geq t.
				\end{cases}
			\end{align*}
			Therefore, we have the stochastic orders
			\begin{align*}
				\left(TB_l\big \vert \vert C_l\vert=n\right)&\stackrel{st}{\leq}\left(TB_l\big \vert \vert C_l\vert=n, TB_l>t\right),\\
				\left(\exp\left(\frac{\theta_0}{2}TB_l\right)\Big\vert \vert C_l\vert=n\right)&\stackrel{st}{\leq}\left(\exp\left(\frac{\theta_0}{2}TB_l\right)\Big\vert \vert C_l\vert=n,TB_l>t\right).
				\label{eq: TB sto domin}
			\end{align*}
			Consequently, 
			$$\EE\left[\exp\left(\frac{\theta_0}{2}TB_l\right)\Bigg\vert \vert C_l\vert=n\right]\leq \EE\left[\exp\left(\frac{\theta_0}{2}TB_l\right)\Bigg\vert \vert C_l\vert=n,TB_l>t\right].$$
			As a result, we have
			\begin{align*}
				&\EE\left[\exp\left(\vert C_l\vert \log (\psi_b(\theta_0/2))\right)\Big\vert TB_l>-t_l^1 \right]\\
				=&\int_{0}^{\infty}\EE\left[\EE\left[\exp\left(\frac{\theta_0}{2}TB_l\right)\Big\vert \vert C_l\vert\right]\Bigg\vert TB_l>t\right]p(-t)dt \\
				=&\int_{0}^{\infty}\sum_{n=1}^{\infty}\EE\left[\exp\left(\frac{\theta_0}{2}TB_l\right)\Big\vert \vert C_l\vert=n \right]\PP(\vert C_l\vert =n\big \vert TB_l>t)p(-t)dt \\
				\leq&\int_{0}^{\infty}\sum_{n=1}^{\infty}\EE\left[\exp\left(\frac{\theta_0}{2}TB_l\right)\Big\vert \vert C_l\vert=n,TB_l>t \right]\PP(\vert C_l\vert =n\big \vert TB_l>t)p(-t)dt\\
				=&\int_0^\infty\EE\left[\exp\left(\frac{\theta_0}{2} TB_l\right)\Big\vert TB_l>t\right]p(-t)dt\\
				= &  \int_0^\infty\EE\left[\exp\left(\frac{\theta_0}{2} TB_l\right)\ind{TB_l>t}\right]\PP(TB_l>t)^{-1}p(-t)dt\\
				\leq& \int_0^\infty\EE[\exp(\theta_0 TB_l)]e^{-\theta_0t/2}\PP(TB_l>t)^{-1}\frac{\PP(TB_l>t)}{\EE[TB_l]}dt
				=\frac{2\EE[\exp(\theta_0 TB_l)]}{\theta_0\EE[TB_l]}.
			\end{align*}
			As a consequence, we can give an upper bound on $\EE[\exp(\theta_0' \bar{J}_0)]$. Recall that $N_{B_0}$ is a Poisson random variable with mean $\lambda_0\EE[TB_l]$, and let's denote the m.g.f of $N_{B_0}$ by $\psi_{N_{B_0}}(\cdot)$. By Poisson thinning, we have
			\begin{align*}
				\EE[\exp(\theta_1 \bar{J}_0)]&\leq \psi_{N_{B_0}}\left(\log \EE\left[\exp\left(\vert C_l\vert \log(\psi_b(\theta_0/2))\right)\Big\vert TB_l>-t_l^1\right]\right)\\
				&\leq \psi_{N_{B_0}} \left(\log\left( \frac{2\EE[\exp(\theta_0TB_l)]}{\theta_0\EE[TB_l]}\right)\right)\\
				&= \exp\left(\frac{2\lambda_0\EE[\exp(\theta_0TB_l)]}{\theta_0}-\lambda_0 \EE[TB_l]\right)\\
			\end{align*}
			Therefore, $\EE[\exp(\theta_1 \bar{J}_0)]<\infty$ as long as $\EE[\exp(\theta_0 TB_l)]<\infty$. Denote the m.g.f of $\vert C_l\vert$ by $\psi_{\vert C_l\vert}$, then
			\begin{align*}
				\EE\left[\exp(\theta_0TB_l)\right]&=\EE\left[\exp\left(\theta_0\sum_{k=1}^{\vert C_l\vert}b_l^k\right)\right]\\
				&=\psi_{\vert C_l\vert}\left(\log \psi_b(\theta_0)\right).
			\end{align*}
			It's known that $\vert C_l\vert$ follows Borel distribution with parameter $m$,
			i.e., $$\PP(\vert C_l\vert = k)=\frac{(mk)^{k-1}e^{-mk}}{k!}$$
			(see \cite{dwass1969total}). As $\log \psi_b(\theta_0)<m-1-\log m$, the moment generating function $\psi_{\vert C_l\vert}(\log \psi_b(\theta_0))<\infty$ by direct calculation.
			As a result, we have 
			\begin{equation}\label{eq: TB bound}
				\EE[\exp(\theta_0 TB_l)]<\infty,\quad \text{and}\quad \EE[\exp(\theta_1 \bar{J}_0)]<\EE\left[\exp\left(\frac{\theta_0}{2}\sum_{l=1}^{N_{B_0}}TB_l\right)\right]<\infty.
			\end{equation}
			This closes our proof.
		\end{enumerate}
	\end{proof}

	\subsection{Proofs of Auxiliary Results in Section \ref{sec: ergodicity}}\label{sec: proof ergo}
	\begin{proof}{Proof of Lemma \ref{lmm: sychronous coupling1}}
		We have 
		$$\vert W^1(t)-W^2(t)\vert=\max\left((W^1(t)-W^2(t))^+,(W^2(t)-W^1(t))^+\right),$$
		where $(x)^+=\max(x,0)$. First, we prove that 
		$$(W^1(t)-W^2(t))^+\leq (W^1(0)-W^2(0))^++J_0^1.$$
		Without loss of generality, for any $t>0$, suppose there are $k_t$ jobs in $J_0^1$ coming at time $0<t_1<\cdots<t_{k_t}<t$ and the corresponding job sizes are $V_i$ for $i=1,\cdots,k_t$. 
		
		We claim that $(W^1(t)-W^2(t))^+$ is monotonically non-increasing between arrival intervals $(0,t_1),(t_1,t_2),\cdots,(t_{k_t-1},t_{k_t}),(t_{k_t},t)$. Take $(t_{k_t},t)$ as an example, and there is no arrival from $J_0^1$ during $(t_{k_t},t)$. So there are only two kinds of arrivals, i.e., the common arrivals and arrivals from $J_0^2$. For common arrivals, since both $W^1(t)$ and $W^2(t)$ jump at them together, they have no influence on the difference $W^1(t)-W^2(t)$. On the other hand, for arrivals from $J_0^2$, $(W^1(t)-W^2(t))^+$ decreases at these jumps. So we can see that $(W^1(t)-W^2(t))^+$ is monotonically non-increasing in $(t_k,t)$. Thus, for $t>t_k$,
		$$(W^1(t)-W^2(t))^+\leq (W^1(t_k)-W^2(t_k))^+.$$
		Following this argument, we have, for $t>t_k$,
		\begin{align*}
			(W^1(t)-W^2(t))^+ &\leq ( W^1(t_{k_t})-W^2(t_{k_t}))^+ \\
			&\leq( W^1(t_{k_t}^-)-W^2(t_{k_t}^-))^+ +V_{k_t}\\
			&\leq ( W^1(0)-W^2(0))^++\sum_{i=1}^{k_t}V_i\\
			&\leq ( W^1(0)-W^2(0))^++J_0^1.
		\end{align*} 
		Similarly, this holds for $(W^2(t)-W^1(t))^+$. So 
		\begin{align*}
			\vert W^1(t)-W^2(t)\vert &\leq \max ((W^1(0)-W^2(0))^+,(W^2(0)-W^1(0))^+)+\max(J_0^1,J_0^2)\\
			&=\vert W^1(0)-W^2(0)\vert+\max(J_0^1,J_0^2).
		\end{align*}
	\end{proof}
	
	\begin{proof}{Proof of Lemma \ref{lmm: sychronous coupling2}}
		By definition of $L_0^i$, after $\max(L_0^1,L_0^2)$, all jobs in $N_0^i$ have already gone. Therefore, the semi-synchronously coupled queues have the same arrival times and job sizes. Without loss of generality, suppose $W_1(\max(L_0^1,L_0^2))\geq W_2(\max(L_0^1,L_0^2))$. Then, because two queues have the same arrival and service times, $W_1(t)\geq W_2(t)$ for any $t\geq \max(L_0^1,L_0^2)$. Therefore, 
		$$\tau_1\geq\tau_2,\quad \text{and} \quad0=W_1(\tau_1)\geq W_2(\tau_1)\geq 0.$$
		As a result, $W_1(\tau_1)=W_2(\tau_1)=0$. In addition, because two queues have the same arrivals and job sizes after $\max(L_0^1,L_0^2)$, 
		$$W^1(t)=W^2(t),\quad \forall~t\geq \tau_1\vee\tau_2.$$ 
	\end{proof}
	
	\begin{proof}{Proof of Proposition \ref{prop: mixing time}}
		We first show that $M_i(t)$ is indeed a super-martingale with respect to $\mathcal{F}_{L_0^i+t}$. Recall that $\hat{R}(L_0^i,t)$ is generated by a compound Poisson process, then by direct calculation, 
		\begin{align*}
			\EE[M_i(t+s)|\mathcal{F}_{L_0^i+t}]&=M_i(t) \EE[\exp(\theta (\hat{R}(L_0^i,t+s)-\hat{R}(L_0^i,t))+\eta\mu t)|\mathcal{F}_{L_0^i+t}]\\
			&=M_i(t) \EE[\exp(\theta\hat{R}(0,s)+\eta\mu s)]\\
			&\leq M_i(t)\EE[\exp(\theta_1\hat{R}(0,s)+\eta_1 \mu s)]^{\theta/\theta_1}\\
			&=M_i(t)\exp\left(\mu s\left[\frac{\lambda_0}{\mu}(\psi_{S_1}(\theta_1)-1)-\theta_1 +\eta_1 \right]\right)^{\theta/\theta_1}\\
			&\leq M_i(t)\exp\left(\mu s\left[\frac{\lambda_0}{\underline{\mu}-{\varepsilon}}(\psi_{S_1}(\theta_1)-1)-\theta_1 +\eta_1 \right]\right)^{\theta/\theta_1}\\
			&=M_i(t)\exp(-\mu s \eta_1/2)^{\theta/\theta_1}\leq M_i(t),
		\end{align*}
		where the first equality holds due to homogeneous independent incremental of Lévy process $\hat{R}(L_0^i,t)$, the first inequality holds since $\theta/\theta_1<1/2$ by definition of $\theta$ and Jensen inequality and the last equality holds due to the choice of $\theta_1$ and $\eta_1$.
		Therefore, by Fatou lemma,
		$$\EE[\exp(\eta \mu \hat{\tau}_i)\vert\mathcal{F}_{L_0^i}]=\EE[M_i(\hat{\tau}_i)\vert\mathcal{F}_{L_0^i} ]\leq \EE[M_i(0)\vert\mathcal{F}_{L_0^i}]=\exp(\theta W^i(L_0^i)+\theta J_0(L_0^i)).$$
		
		Based on this result, we next bound $\PP(\tau_i>t\vert\mathcal{F}_0)$.
		We have the following upper bounds for $W^i(L_0^i)$ and $J_0(L_0^i)$
		\begin{align*}
			W^i(L_0^i)&\leq W^i(0)+J_0^i(0)+\sum_{0\leq t_l^1\leq L_0^i}S_l=W^i(0)+J_0^i(0)+\hat{R}(0,L_0^i)+\mu L_0^i,\\
			J_0(L_0^i)&\leq \sum_{0\leq t_l^1\leq L_0^i}S_l=\hat{R}(0,L_0^i)+\mu L_0^i.
		\end{align*}
		Therefore,
		\begin{align*}
			\PP(\tau_i>t\vert \mathcal{F}_0)&\leq e^{-\eta\mu t}\EE[e^{\eta\mu \tau_i}\vert \mathcal{F}_0]\\
			&\leq e^{-\eta\mu t}\EE[e^{\eta\mu (\hat{\tau}_i+ L_0^i)}\vert \mathcal{F}_0]\\
			&\leq e^{-\eta\mu t}e^{\eta \mu L_0^i}\EE[\exp(\theta W^i(L_0^i)+\theta J_0(L_0^i))\vert\mathcal{F}_0]\\
			&\leq e^{-\eta \mu t}e^{(\eta+2\theta)\mu L_0^i}e^{\theta(W^i(0)+J_0^i(0))}\EE[\exp(2\theta\hat{R}(0,L_0^i))\vert\mathcal{F}_0]\\
		\end{align*}
		Note that $\theta<\theta_1/2, \eta<\eta_1/2$, we have that $\exp(2\theta\hat{R}(0,t)+2\eta\mu t)$ is a positive super-martingale following the proof of $M_i(t)$. Therefore,
		$$\EE[\exp(2\theta\hat{R}(0,L_0^i))\vert \mathcal{F}_0]\leq \EE[\exp(2\theta\hat{R}(0,L_0^i)+2\eta\mu L_0^i)\vert \mathcal{F}_0]\leq 1.$$
		As a result,
		$$\PP(\tau_i>t\vert\mathcal{F}_0)\leq e^{-\eta \mu t}e^{(\eta+2\theta)\mu L_0^i}e^{\theta(W^i(0)+J_0^i(0))}.$$
	\end{proof}

	\subsection{Proofs of the Auxiliary Results in Section \ref{subsec: online algorithm} and \ref{subsec: regret analysis}}\label{subsec: Proofs of regret analysis}
	In this section, we first prove the gradient representation result in Proposition \ref{prop: gradient design}. Then, we prove three critical properties: (i) boundedness; (ii) exponential ergodicity of Hawkes queues; and (iii) convexity in Section \ref{subsec: regret analysis}. Finally, we give a detailed proof of the regret analysis in Section \ref{subsec: proofs regret}.
	\subsubsection{Proof of Proposition \ref{prop: gradient design}}
	We first introduce an intermediate result, which describes the workload difference of two \textbf{fully synchronously coupled} $G/G/1$ queue pasewisely. The proof of Lemma \ref{lmm: Lipchitz} is omitted and we refer the readers to \cite{onlinequeue2020}. 
	\begin{lemma}[Pathwise Lipchitz Lemma (Lemma 3 of \cite{onlinequeue2020})]  For any two $G/G/1$ queues (general input) with initial states $W^1(0)$ and $W^2(0)$, if they share the same arrival process $N(t)$ and the same load $V_i$'s, but have the different rates $\mu_1,\mu_2$, then the difference of the workload processes are bounded by 
		$$\vert W^1(t)-W^2(t)\vert\leq \vert W^1(0)-W^2(0)\vert+\vert \mu_1-\mu_2\vert \max (X^1(0),X^2(0)).$$
		\label{lmm: Lipchitz}
	\end{lemma}
	
	\begin{proof}{Proof of Proposition \ref{prop: gradient design}}
		The proof of Proposition \ref{prop: gradient design} is based on our construction of stationary workload $\tilde{W}(t)$ in Section 3. To distinguish the workload under different $\mu$, in this proof, we denote $R^\leftarrow_\mu(0,t)=\sum_{N(-t)}^{-1}V_i-\mu t$ and $\tilde{W}_\mu(0)=\max_{u\geq 0} R^\leftarrow_\mu(0,u)$.
		Recall that $\tilde{W}_\mu (t)$
		follows the stationary distribution of the workload process, and the corresponding stationary busy time is
		$$\tilde{X}_\mu(t)=\arg\max_{u\geq 0}R^\leftarrow_\mu (t,u).$$
		Therefore, by taking the derivative with respect to $\mu$ pathwisely, we have 
		\begin{align*}
			\tilde{W}_\mu(0)'&=\left(\max_{u\geq 0}R_\mu^\leftarrow(0,u)\right)'\\
			&=\left(\max_{u\geq 0}\sum_{N(-u)}^{-1}V_i-\mu u\right)'\\
			&=-\arg\max_{u\geq 0}R_{\mu}^\leftarrow(0,u)=-\tilde{X}(0),
		\end{align*}
		with probability one. If we could interchange the derivative and the expectation, we could have 
		$$w'(\mu)=\EE[\tilde{W}_\mu(0)]'=\EE[\tilde{W}_\mu(0)']=-\EE[\tilde{X}_\mu(0)].$$
		
		We justify the interchange of expectation and derivative by showing that $\tilde{W}_\mu(0)$ is Lipchitz continuous with respect to $\mu$ (by Lemma \ref{lmm: Lipchitz}).
		Note that, the workload and observed busy time $\tilde{W}_\mu(t),\tilde{X}_\mu(t)$ are non-increasing with $\mu$, i.e., for any $\mu_1<\mu_2$ and $t\leq 0$,
		$$\tilde{W}_{\mu_1}(t)\geq\tilde{W}_{\mu_2}(t),\quad \tilde{X}_{\mu_1}(t)\geq\tilde{X}_{\mu_2}(t).$$
		Let $t_0=-\tilde{X}_{\underline{\mu}}(0)$. Then, for $\mu\in\mathcal{B}$, $\mu\geq\underline{\mu}$, and consequently,
		$$0=\tilde{W}_{\underline{\mu}}(t_0)\geq \tilde{W}_\mu(t_0)\geq 0.$$
		As a result, by Lipchitz lemma (Lemma \ref{lmm: Lipchitz}),
		\begin{align*}
			\vert \tilde{W}_{\mu+h}(0)-\tilde{W}_{\mu}(0)\vert&\leq \vert\tilde{W}_{\mu+h}(t_0) -\tilde{W}_{\mu}(t_0) \vert+h \max (\tilde{X}_{\mu+h}(0),\tilde{X}_{\mu}(0))\\
			\leq &0+ h\tilde{X}_{\underline{\mu}}(0).
		\end{align*}
		Then, by Proposition \ref{thm: moment of W}, $$\EE[\tilde{X}_{\underline{\mu}}(0)]<\infty,$$  so we know the expectation and derivatives can be interchanged by dominate convergence theorem.
		
		For Lipchitz property of $w(u)$, notice that 
		$$\vert w(u_1)-w(u_2)\vert=\vert \EE[\tilde{W}_{\mu_1}(0)-\tilde{W}_{\mu_2}(0)]\vert \leq \EE[\tilde{X}_{\underline{\mu}}(0)]\vert \mu_1-\mu_2\vert.$$
		This closes the proof.
	\end{proof}  
	\subsubsection{Proof of three key properties in Section \ref{subsec: regret analysis}}
	In this section, we prove the three key properties in Section \ref{subsec: regret analysis}: (i) boundedness; (ii) exponential ergodicity; (iii) convexity and smoothness.
	\begin{proof}{Proof of Corollary \ref{lmm:bounded}}
		First, recall that we assume the Hawkes$/GI/1$ queue  is operated from an empty state. For any control sequence, we couple it with a dominant system always under control $\mu_k=\underline{\mu}$. We couple two systems in a fully synchronous way, i.e., they have the same arrivals and job sizes. Moreover, denote the workload of the dominant system by $\bar{W}_k(t)$, we assume that the dominant system has a stationary initial $(N_{0,1}(0),\bar{W}_1(0))$. Then for any control policy, $\bar{W}_k(t)\geq_{st}W_k(t)$ stochastically and $\bar{W}_k(t)$ also follows the stationary distribtuion of $W_\infty(\underline{\mu})$. By Assumption \ref{assmpt: stable} and Theorem \ref{thm: moment of W}, 
		$$\EE[W_k(t)^{2m}]\leq \EE[\bar{W}_k(t)^{2m}]< \infty.$$
		Similarly, by Theorem \ref{thm: moment of W}, the corresponding system busy time $X^k(t)\leq \bar{X}^k(t)$, and $\EE[X_k(t)^{2m}]\leq \EE[\bar{X}_k(t)^{2m}]<\infty$.
		
		Next, because $\theta\leq \theta_1/6$ by definition, so by Proposition \ref{prop: J0} and Theorem \ref{thm: moment of W}, 
		$$\EE[\exp(6\theta \bar{W}_k(t))],\EE[\exp(6\theta J_{0,k}(t))],\EE[J_{0,k}(t)^{2m}]<\infty.$$ 
		Finally, for $L_{0,k}$, it's the residual life of $N_{0,k}(0)$ and it can be bounded by the total lifetime of $N_{B_0,k}$, i.e.
		$$L_{0,k}\leq \sum_{i=1}^{N_{B_0,k}}TB_i.$$
		In \eqref{eq: TB bound} , we have already proved that 
		$$\EE\left[\exp\left(\frac{\theta_0}{2}\sum_{i=1}^{N_{B_0,k}}TB_i\right)\right]=\EE\left[\EE\left[\exp\left(\frac{\theta_0}{2}TB_l\right)\Big\vert TB_l>-t_l^1\right]^{N_{B_0,k}}\right]<\infty.$$
		Now 
		$$6(\eta+2\theta)\mu_k\theta= \frac{\mu_k}{\max(\bar{\mu},1)}\min(\theta_0/2,\theta_1)\leq \frac{\theta_0}{2},$$
		and consequently, 
		$$\EE[\exp(6(\eta+2\theta))\mu_k L_{0,k}]<\infty.$$
		Since all of these queueing functions are finite, we can finish the proof of this lemma by choosing 
		\begin{align*}
			M=&\max \Big(\EE[W_k(t)^{2m}],\EE[X_k(t)^{2m}], \EE[J_{0,k}(t)^{2m}],\\
			&\quad\quad \quad\EE[e^{6(\eta+2\theta)\bar{\mu} L_{0,k}}],\EE[e^{6\theta W_k(t)}],\EE[e^{6\theta J_{0,k}(t)}]\Big).
		\end{align*}
	\end{proof}
	\begin{proof}{Proof of Corollary \ref{coro: ergodicity compact}}
		Denote 
		\begin{align*}
			e_k&=\exp\left((\eta+2\theta)\mu_k L_{0,k}+\theta(W_k(0)+J_{0,k})\right)\\
			\tilde{e}_k&=\exp\left((\eta+2\theta)\mu_k \tilde{L}_{0,k}+\theta(\tilde{W}_k(0)+\tilde{J}_{0,k})\right),
		\end{align*}
		where $\tilde{L}_{0,k}$ and $\tilde{J}_{0,k}$ are the residual life and total residual job size of $\tilde{N}_{0,k}(0)$. In addition, let $\mathcal{F}_{0,k}=\sigma(W_k(0),\tilde{W}_k(0),N_{0,k}(0),\tilde{N}_{0,k}(0),\mu_k)$. 
		By Theorem \ref{thm: geometric ergodicity}, we have 
		\begin{align*}
			\EE\left[\vert W_k(t)-\tilde{W}_k(t)\vert^m\Big\vert\mathcal{F}_{0,k}\right]\leq \exp\left(-\eta\mu_k t\right)D_0^m(e_k+\tilde{e}_k),
		\end{align*}
		where $D_0=\vert W_k(0)-\tilde{W}_k(0)\vert+J_{0,k}+\tilde{J}_{0,k}.$
		By H\"older inequality, we have
		\begin{align*}
			\EE[D_0^m e_k]=&\EE\left[D_0^me^{(\eta+2\theta)\mu L_{0,k}}e^{\theta W_k(0)}e^{\theta J_{0,k}}\right]\\
			=&\EE\left[(\vert W_k(0)-\tilde{W}_k(0)\vert +J_{0,k}+\tilde{J}_{0,k})^me^{(\eta+2\theta)\mu L_{0,k}}e^{\theta W_k(0)}e^{\theta J_{0,k}}\right]\\
			\leq &\EE\left[(\vert W_k(0)-\tilde{W}_k(0)\vert +J_{0,k}+\tilde{J}_{0,k})^{2m}\right]^\frac{1}{2}\EE[e^{6(\eta+2\theta)\mu L_{0,k}}]^\frac{1}{6}\EE[e^{6\theta W_k(0)}]^\frac{1}{6}\EE[e^{6\theta J_{0,k}}]^\frac{1}{6}\\
			\leq &\left(4^{2m-1}(\EE[W_k(0)^{2m}]+\EE[\tilde{W}_k(0)^{2m}]+\EE[(J_{0,k})^{2m}]+\EE[(\tilde{J}_{0,k}^2)^{2m}])\right)^\frac{1}{2}\sqrt{M}\\
			\leq &4^{m}M
		\end{align*}
		Following the same conduction, $\EE[D_0^m\tilde{e}_k]\leq 4^m M$. Therefore, 
		$$\EE\left[\vert W_k(t)-\tilde{W}_k(t)\vert^m\right]\leq 32M e^{-\eta \underline{\mu}t}$$ as $m\leq 2$. \\
		Similarly, for the second argument, by Theorem \ref{thm: geometric ergodicity} (ergodicity) and H\"older inequality, 
		\begin{align*}
			\EE[\vert X_k(t)-\tilde{X}_k(t)\vert^m]\leq &e^{-\eta \underline{\mu} t} \EE\left[\left(X_k(0)+\tilde{X}_k(0)+t\right)^m(e_k+\tilde{e}_k)\right]\\
			\leq &e^{-\eta \underline{\mu} t}\EE\left[\left(X_k(0)+\tilde{X}_k(0)+t\right)^{2m}\right]^\frac{1}{2}2\sqrt{M}\\
			\leq &2\sqrt{M}e^{-\eta \underline{\mu} t}\left(\EE[X_k(0)^{2m}]+\EE[\tilde{X}_k(0)^{2m}]+t^4\right)^\frac{1}{2}\\
			\leq &\left(2\sqrt{M}(2M+t^4)^\frac{1}{2}\right)e^{-\eta \underline{\mu} t}\leq  (2\sqrt{2}M+2t^2\sqrt{M}) e^{-\eta \underline{\mu} t}.
		\end{align*}
		The last inequality is due to $\sqrt{x+y}\leq \sqrt{x}+\sqrt{y}.$  Then, because $t^2e^{-0.5\eta\underline{\mu} t}\leq C\equiv e^{-2}\log^2(4/\eta\underline{\mu})$, by Theorem \ref{thm: geometric ergodicity},
		\begin{equation*}
			\EE\left[\vert X_k(t)-\tilde{X}_k(t)\vert ^m\right]\leq e^{-\eta \underline{\mu} t}(2\sqrt{2}M+2t^2\sqrt{M})\leq e^{-0.5\eta\underline{\mu} t}(2\sqrt{2}M+2C\sqrt{M}).
		\end{equation*}
		Therefore, by choosing $A=\max(32M,2\sqrt{2}M+2C\sqrt{M})$, we have the result.
	\end{proof}
	\begin{proof}{Proof of Corollary \ref{coro: convex and smooth}}
		We need to prove that $f(\mu)$ is strongly-convex and smooth in the compact region $\mathcal{B}$. By Proposition \ref{prop: gradient design},
		$$\nabla f(\mu)=-h_0\EE[X_\infty(\mu)]+c'(\mu).$$
		Note that the observed busy time $\EE[X_\infty(\mu)]$ is strictly decreasing in $\mu$, so $\EE[X_\infty(\mu)]'<0$ for all $\mu\in\mathcal{B}$. As a result, for all $\mu\in\mathcal{B}$,
		$$\nabla^2 f(\mu)=-h_0\EE[X_\infty(\mu)]'+c''(\mu)>K_0\equiv \min_{\mu\in\mathcal{B}} \{-h_0 \EE[X_\infty(\mu)]'+c''(\mu)\}\wedge 1>0.$$
		Then, by Taylor expansion,  for some $\zeta$ between $\mu$ and $\mu^*$,
		\begin{align*}
			&f(\mu^*)-f(\mu)=\nabla f(\mu)(\mu^*-\mu)+\nabla^2 f(\zeta)(\mu-\mu^*)^2\leq0\\
			\Rightarrow&\nabla f(\mu)(\mu-\mu^*)\geq \nabla^2 f(\zeta)(\mu-\mu^*)^2\geq K_0(\mu-\mu^*)^2.
		\end{align*}
		Similarly, for the second property, let 
		$$K_1\equiv \max_{\mu\in\mathcal{B}}\{-h_0\EE[X_\infty(\mu)]'+c''(\mu)\}.$$
		We have, for some $\zeta'$ between $\mu$ and $\mu^*$,
		$$\vert\nabla f(\mu)\vert=\vert\nabla f(\mu)-\nabla f(\mu^*)\vert=\vert\nabla^2 f(\zeta')\vert\cdot \vert \mu-\mu^*\vert\leq K_1\vert \mu-\mu^*\vert, $$
		as $\nabla f(\mu^*)=0$ by definition.
		This closes the proof.
	\end{proof}
	\subsubsection{Proofs for the Auxiliary Results in Section \ref{subsec: regret analysis}}\label{subsec: proofs regret}
	In this section, we give the full proof of our regret upper bound \ref{thm: main} based on the regret decomposition. We first summarize regret upper bounds to $R_1(L)$ and $R_2(L)$ by Proposition \ref{prop: R1} and \ref{coro: R2}, then give the detailed proof of them.
	\begin{proposition}[Bound for the Regret of nonstationarity]\label{prop: R1}
		Suppose that Assumption \ref{assmpt: stable} to \ref{assum: convex of cmu} hold. If we choose hyperparameters according to Theorem \ref{thm: main}, then there exists a constant $C_1>0$ (specified in \eqref{eq: C1}) such that
		$$R_1(L) \leq C_1\log(L)^2.$$
	\end{proposition}
	Then we turn to the regret of suboptimality $R_2$ and we have a similar result as in Proposition \ref{prop: R1}. 
	\begin{proposition}[Bound for the Regret Due to Suboptimality]\label{coro: R2}
		Suppose Assumption \ref{assmpt: stable} to \ref{assum: convex of cmu} hold. If we select hyperparameters $\eta_k$ and $T_k$ suggested by Theorem \ref{thm: main}, there exists a constant $C_2$ (specified in \eqref{eq: C2}) such that 
		$$R_2(L) \leq  C_2\log(L)^2.$$
	\end{proposition}
	With Proposition \ref{prop: R1} and Proposition \ref{coro: R2}, the regret bound is already at hand.
	\begin{proof}{Proof of Theorem \ref{thm: main}}
		The regret bound holds immediately following Proposition \ref{prop: R1} and Corollary \ref{coro: R2}.
		For $R_1$, we have 
		$$R_1(L)\leq C_1 \log(L)^2.$$
		For $R_2$, we have 
		$$R_2(L)\leq C_2\log(L)^2.$$
		Therefore, in total, we have 
		$$R(L)\leq C_{alg} \log(L)^2,$$
		with 
		\begin{equation}\label{eq: K_alg}
			C_{alg}=C_1+C_2,
		\end{equation} 	
		where $C_1$ and $C_2$ are constants specified in \eqref{eq: C1} and \eqref{eq: C2}.
		Moreover, $M_L=a_TL+c_T\sum_{k=1}^{L}\log(k)=O(L\log(L))$. So $O(\log(L))=O(\log(M_L))$, and 
		$$R(L)=K_{alg}\log(L)^2=O(\log(M_L)^2).$$
		This closes the proof. 
	\end{proof}
	We give the proof of Proposition \ref{prop: R1} and Proposition \ref{coro: R2}, which give upper bounds of regret of transient behaviors and of suboptimality respectively. We first give the proof of the upper bound of $R_1(L)$.
	\begin{proof}{Proof of Proposition \ref{prop: R1}}
		$R_1(L)$ is the regret due to the transient behavior of Hawkes queues. We analyze $R_{1,k}$ by decomposing each cycle into two periods: (i) warm-up period $(0,\frac{1}{2}T_k)$; (ii) near-stationary period $(\frac{1}{2}T_k,T_k)$. In the warm-up period, the workload $W_k(0)$ is close to the steady state of cycle $k-1$, i.e., $W_\infty(\mu_{k-1})$; on the other hand, the workload in the near-stationary period is close to the steady state of cycle $k$, i.e., $W_\infty(\mu_k)$. So in the warm-up period, we compare $W_k(t)$ with $W_\infty(\mu_{k-1})$, and in the near-stationary period, we contrast $W_k(t)$ with $W_\infty(\mu_k)$.
		
		Specifically, we need to construct an appropriate stationary workload process semi-synchronously coupled with $W_k(t)$. For this purpose, at the beginning of each cycle $k$, we draw $(\tilde{W}_k(0),\tilde{J}_{0,k}(0))$ jointly from the stationary distribution of $(W_k(\infty),J_0)$ independently. Then we couple $\tilde{W}_k(t)$ with $W_k(t)$ semi-synchronously, i.e., $\tilde{W}_k(t)$ and $W_k(t)$ share the same coming clusters after $t=0$. Nevertheless, the residual jobs $J_{0,k}(0)$ and $\tilde{J}_{0,k}(0)$ differ. In this way, $\tilde{W}_k(t)$ is in steady state. In addition, we extend $\tilde{W}_{k}(t)$ to $t\in[0,T_k+T_{k+1}]$ so that $W_k(T_k+s)$ also share the same arrivals and loads with $W_{k+1}(s)$ for $s\in[0,T_{k+1}]$ (except those from $\tilde{N}_{0,k}(0)$ and $N_{0,k}$).
		
		
		Following this idea, we have 
		\begin{align*}
			&\left\vert \EE\left[\int_{0}^{T_k}W_k(t)-w(\mu_k)dt\right]\right\vert \\
			\leq &\underbrace{\left\vert \EE\left[\int_{0}^{\frac{1}{2}T_k}W_k(t)-w(\mu_k)dt\right]\right\vert }_{I_1}+\underbrace{\left\vert \EE\left[\int_{\frac{1}{2} T_k}^{T_k}W_k(t)-w(\mu_k)dt\right]\right\vert }_{I_2}
		\end{align*}
		For $I_2$, Corollary \ref{coro: ergodicity compact} yields
		\begin{align*}
			I_2&=\left\vert\EE\left[\int_{\frac{1}{2}T_k}^{T_k}W_k(t)-\tilde{W}_k(t)\right]\right\vert\\
			&\leq \int_{\frac{1}{2} T_k}^{T_k}\EE\left[\vert W_k(t)-\tilde{W}_k(t)\vert \right]dt\leq A\int_{\frac{1}{2} T_k}^{T_k}e^{-\eta \underline{\mu}t}dt\\
			&\leq \frac{A}{\eta\underline{\mu}}e^{-\frac{\eta}{2}\underline{\mu} T_k}\leq \frac{K_0}{8\eta\underline{\mu}}k^{-1}.
		\end{align*}
		For $I_1$, we intend to compare $W_k(t)$ with $\tilde{W}_{k-1}(T_{k-1}+t)$ using pathwise Lipchitz lemma, i.e., Lemma \ref{lmm: Lipchitz}. However, this lemma requires that $W_k(t)$ and $\tilde{W}_{k-1}(T_{k-1}+t)$ are fully synchronously coupled, which happens when residual jobs $J_{0,k-1}(0),\tilde{J}_{0,k-1}(0)$ have all gone at the beginning of cycle $k$, i.e., when $\max(L_{0,k-1},\tilde{L}_{0,k-1})<T_{k-1}$. Therefore, we further decompose $I_1$ into two terms.
		
		\begin{align*}
			I_1\leq &\underbrace{\left\vert \int_{0}^{\frac{1}{2} T_k}\EE\left[\left(W_k(t)-w(\mu_k)\right)\ind{\max(L_{0,k-1},\tilde{L}_{0,k-1})>T_{k-1}}\right] dt \right\vert}_{B_1} \\
			&+\underbrace{\left\vert \int_{0}^{\frac{1}{2} T_k}\EE\left[\left(W_k(t)-w(\mu_k)\right)\ind{\max(L_{0,k-1},\tilde{L}_{0,k-1})\leq T_{k-1}}\right] dt\right\vert}_{B_2}. 
		\end{align*}
		For $B_1$, because $T_k$ grows to $\infty$, $\PP(\max(L_{0,k-1},\tilde{L}_{0,k-1})>T_{k-1})$ is small for large enough $k$. So,
		\begin{align*}
			B_1=&\left\vert\EE\left[\left( W_k(t)-w(\mu_k) \right) \ind{\max(L_{0,k-1},\tilde{L}_{0,k-1})>T_{k-1}}\right]\right\vert\\
			\stackrel{(a)}{\leq }& \left(\EE[\vert W_k(t)-\tilde{W}_k(t)\vert ^2]\right)^\frac{1}{2}\PP\left(\max(L_{0,k-1},\tilde{L}_{0,k-1})>T_{k-1}\right)^\frac{1}{2}\\
			\stackrel{(b)}{\leq}& 2\sqrt{M}\left(2\EE[e^{\eta\bar{\mu} L_{0,k-1}}]e^{-\eta\underline{\mu} T_{k-1}}\right)^{\frac{1}{2}}\stackrel{(c)}{\leq}  2\sqrt{2}Me^{-0.5\eta\underline{\mu} T_{k-1}}\\
			\stackrel{(d)}{\leq}&2\sqrt{2}M(k-1)^{-1}\leq 4\sqrt{2}Mk^{-1}.
		\end{align*}
		Here inequality $(a)$ is Cauchy-Schwartz inequality. Inequality $(b)$ and $(c)$ comes from the boundedness property in Corollary \ref{lmm:bounded}. The inequality $(d)$ holds because of our choice of $T_k$.
		
		For $B_2$, we first give a bound on $\vert\EE [w(\mu_k)-w(\mu_{k-1})]\vert$. By Proposition \ref{prop: gradient design}, 
		\begin{align*}
			\vert \EE[w(\mu_k)-w(\mu_{k-1})]\vert &\leq \EE[X_\infty(\underline{\mu})]\EE[\vert\mu_k-\mu_{k-1} \vert]\leq M\eta_{k-1} \EE[\vert H_{k-1}\vert]\\
			&\leq Mc_\eta \left(((1-\xi)T_{k-1})^{-1}\int_{\xi T_{k-1}}^{T_{k-1}}\EE[\vert X_{k-1}(t)\vert]+c'(\mu_{k-1})dt\right)(k-1)^{-1}\\
			&\leq Mc_\eta \underbrace{(M+\max_{\mu\in\mathcal{B}}c'(\mu))}_{\equiv M_1}(k-1)^{-1}=c_\eta M M_1 (k-1)^{-1}.
		\end{align*}
		Similarly, we also have 
		\begin{equation}\label{eq: eta_kV_k}
			\EE[(\mu_k-\mu_{k-1})^2]=\eta_k^2\EE[H_k^2]\leq 2c_\eta^2 k^{-2}\cdot(2\EE[X_\infty(\underline{\mu})^2]+2\max_{\mu\in\mathcal{B}}c'(\mu)^2)\equiv 2c_\eta^2M_2 k^{-2}.
		\end{equation}
		Then $B_2$ can be bounded as follows.
		\begin{align*}
			B_2=&\left\vert \EE\left[\left(W_k(t)-w(\mu_k)]\right)\ind{\max(L_{0,k-1},\tilde{L}_{0,k-1})\leq T_{k-1}}\right]\right\vert \\
			\leq& \vert\EE[ w(\mu_k)-w(\mu_{k-1})]\vert +\left\vert \EE\left[\left(W_k(t)-w(\mu_{k-1})\right)\ind{\max(L_{0,k-1},\tilde{L}_{0,k-1})\leq T_{k-1}}\right]\right\vert \\
			\stackrel{(a)}{\leq}& c_\eta MM_1(k-1)^{-1}+\EE\left[\vert W_k(t)-\tilde{W}_{k-1}(T_{k-1}+t)\vert \ind{\max(L_{0,k-1},\tilde{L}_{0,k-1})\leq T_{k-1}} \right]\\
			\stackrel{(b)}{\leq } &Mc_\eta M_1\underbrace{(k-1)^{-1}}_{\leq 2 k^{-1}}+\EE[ \underbrace{\vert W_k(0)-\tilde{W}_{k-1}(T_{k-1})\vert }_{=\vert W_{k-1}(T_{k-1})-\tilde{W}_{k-1}(T_{k-1})\vert }]\\
			&\hspace{5cm}+\EE\left[\vert \mu_k-\mu_{k-1}\vert \max(X_k(0),\tilde{X}_{k-1}(T_{k-1}))\right]\\
			\stackrel{(c)}{\leq} &c_\eta M M_1\cdot2k^{-1}+Ae^{-\eta\underline{\mu}T_{k-1}}+\EE[\vert \mu_k-\mu_{k-1}\vert ^2]^{\frac{1}{2}}\EE[(X_k(0)+\tilde{X}_{k-1}(T_{k-1}))^2]^\frac{1}{2}\\
			\stackrel{(d)}{\leq} &(2c_\eta M M_1+A+2c_\eta\sqrt{2MM_2})k^{-1}. 
		\end{align*}
		Here inequality $(a)$ holds because of the triangle inequality and the fact that $\tilde{W}_{k-1}(T_{k-1}+t)$ is in steady-state. Inequality $(b)$ follows from the Lipchitz lemma (Lemma \ref{lmm: Lipchitz}). Inequality $(c)$ is the direct result of ergodicity (Corollary \ref{coro: ergodicity compact}). The inequality $(d)$ uses the boundedness of $\EE[X_k^2(0)]$ and $\EE[\tilde{X}^2_{k-1}(T_{k-1})]$ and \eqref{eq: eta_kV_k}. To sum up, we have
		\begin{align*}
			I_1\leq \frac{1}{2} T_kk^{-1} (2c_\eta MM_1+A+2c_\eta\sqrt{2MM_2}+4\sqrt{2}M)=O(\log(k)k^{-1})
		\end{align*} 
		since we choose $T_k=O(\log(k))$.\par 
		In total, we have 
		\begin{align*}
			I_1&\leq\frac{1}{2}(a_T+c_T)(2c_\eta MM_1+A+2c_\eta\sqrt{2MM_2}+4\sqrt{2}M) \log(k) k^{-1}=O(\log(k)k^{-1})\\
			I_2&\leq \frac{K_0}{8\eta \underline{\mu}}k^{-1}=O(k^{-1}).
		\end{align*}
		As a result,
		$$R_1(L)\leq C_1 \log (L)^2$$
		with
		\begin{equation}\label{eq: C1} 
			C_1=\frac{1}{2}(a_T+c_T)(2c_\eta MM_1+A+2c_\eta\sqrt{2MM_2}+4\sqrt{2}M)+ \frac{K_0}{8\eta \underline{\mu}}.
		\end{equation}
	\end{proof}
	
	Now we turn to the proof of $R_2(L)$.	Specifically, $R_2$ relies only on the ``goodness" of the gradient approximation $H_k$. Below we define two quantities to measure the ``goodness" of $H_k$ as the gradient estimator.
	\begin{equation*}
		B_k\equiv\max_{\mu_k\in\mathcal{B}}\left\vert\EE\left[ H_k-\nabla f(\mu_k) \big\vert \mu_k\right]\right\vert,\quad \mathcal{V}_k\equiv\max_{\mu_k\in\mathcal{B}}\EE\left[H_k^2\big\vert  \mu_k\right].
	\end{equation*}
	The following Proposition \ref{prop: chen R2}, which appears in Theorem 2 of \cite{onlinequeue2020}, is an extension of conventional SGD convergence result (e.g., \cite{kushner2003stochastic}) allowing biased gradient estimator and unequal cycle length. We provide the analysis of $R_2$ with Proposition \ref{prop: chen R2} used as a black-box.
	
	\begin{proposition}[Theorem 2 in \cite{onlinequeue2020} ]\label{prop: chen R2}
		Suppose the objective function $f(\mu)$ is strongly convex and smooth in $\mathcal{B}$, i.e., there exists $K_0\leq 1$ and $K_1>K_0$, such that
		\begin{enumerate}
			\item $(\mu-\mu^*)\nabla f(\mu)\geq K_0(\mu-\mu^*)^2$
			\item $\vert \nabla f(\mu)\vert \leq K_1 \vert \mu-\mu^*\vert $.
		\end{enumerate}
		In addition, if there exists a constant $K_3\geq 1$ and $k_0>0$ such that the following conditions holds for all $k>k_0$,
		\begin{enumerate}
			\item[(a)]$\frac{1}{k}\leq \frac{K_0}{2}\eta_k$,
			\item[(b)] $B_k\leq \frac{K_0}{8}k^{-1}$,
			\item[(c)] $\eta_k\mathcal{V}_k\leq K_3k^{-1}$,
		\end{enumerate}
		then, there exists a constant $K_4\geq 8K_3/K_0$ such that for all $k\geq 1$, 
		$$\EE[(\mu_k-\mu^*)^2]\leq K_4k^{{-1}},$$
		and as a consequence,
		$$R_2(L)\leq K_1K_4\sum_{k=1}^{L}T_kk^{-1}.$$
	\end{proposition}
	
	\begin{proof}{Proof of Theorem \ref{coro: R2}}
		We verify conditions $(a)$ to $(c)$ one by one.
		For condition (a), notice that $\eta_k=c_\eta k^{-1}$ with $c_\eta\geq \frac{2}{K_0}$, condition $(a)$ holds automatically.
		
		For condition (b), recall that the gradient estimator is 
		$$H_k=-\frac{h_0}{(1-\xi)T_k}\int_{\xi T_k}^{T_k}X_k(t)dt+c'(\mu).$$
		Therefore, the bias can be bounded by 
		\begin{align*}
			\EE\left[\left \vert\frac{1}{(1-\xi)T_k}\int_{\xi T_k}^{T_k}X_k(t)-\EE[X_k(\infty)]\right\vert\right]&\leq \EE\left[\frac{1}{(1-\xi)T_k}\int_{\xi T_k}^{T_k}\vert X_k(t)-\tilde{X}_k(t)\vert\right]\\
			&\leq A\exp(-0.5\eta \underline{\mu}\xi T_k)\leq \frac{K_0}{8}k^{-1},
		\end{align*}
		with $T_k=a_T+c_T\log(k)$, $c_T\geq 2(\eta \underline{\mu}\xi )^{-1}\log(8A/K_0)$.
		
		Finally, for condition (c), notice that $\eta_k=c_\eta k^{-1}$, and we only need to show that $\EE[H_k^2]<\infty$. Similar to \eqref{eq: eta_kV_k}, we have 
		\begin{align*}
			\EE[H_k^2]&=\EE\left[\left(((1-\xi)T_{k-1})^{-1}\int_{\xi T_{k-1}}^{T_{k-1}}\EE[\vert X_{k-1}(t)\vert]+c'(\mu_{k-1})\right)^2\right]\\
			&\leq 2\EE[X_{\infty}(\underline{\mu})^2]+2\max_{\mu\in\mathcal{B}}c'(\mu)^2=M_2<\infty,
		\end{align*}
		where $M_2$ is defined in \eqref{eq: eta_kV_k}.
		Therefore, take $K_3=M_2\vee 1$ and the condition $(c)$ holds.
		Therefore, by choosing 
		\begin{equation}\label{eq: C2}
			C_2=K_4K_1(a_T+c_{T})
		\end{equation}
		we have 
		$$R_2(L)\leq K_4K_1\sum_{k=1}^{L}(a_T+c_T)k^{-1}\log k\leq C_2 \log (L)^2.$$
	\end{proof}

\end{document}